[1994/12/01]
\documentclass{ijmart-mod}
\chardef\bslash=`\\ 

\usepackage[T2A,T1]{fontenc}
\newcommand{\lob}{\mbox{\usefont{T2A}{\rmdefault}{m}{n}\cyrl}}

\usepackage{bm}
\usepackage{graphicx}
\usepackage[breaklinks=true]{hyperref}
\usepackage{mathtools}
\usepackage{caption}
\usepackage{amsmath}
\usepackage{array}
\usepackage{multirow}
\usepackage{soul}

\newtheorem{thm}{Theorem}[section]

\newtheorem{lem}[thm]{Lemma}
\newtheorem{prop}[thm]{Proposition}

\theoremstyle{definition}
\newtheorem{defn}[thm]{Definition}
\newtheorem{rem}[thm]{Remark}

\theoremstyle{remark}

\newcommand{\eval}[2][\right]{\relax
  \ifx#1\right\relax \left.\fi#2#1\rvert}

\begin{document}

\title{Flip-graph non-convexity for once-punctured~polygons}

\author[L. Pournin]{\textup{Lionel Pournin}}
\address{Universit{\'e} Paris 13, Villetaneuse, France}
\email{lionel.pournin@univ-paris13.fr}

\author[Z. Wang]{\textup{Zili Wang}}
\address{Sun Yat-sen University, Shenzhen, Guangdong, China}
\email{wangzli6@mail.sysu.edu.cn} 

\maketitle

\begin{abstract}
The set of the triangulations with vertex set $X$ of a simple polygon $\mathrm{P}$ can be structured into a flip-graph $\mathcal{F}(\mathrm{P},X)$ whose edges connect two triangulations that differ by a single arc. The geometry of flip-graphs has been thoroughly studied and it is known that the subgraph $\mathcal{F}_\varepsilon(\mathrm{P},X)$ induced by the triangulations that contain a given arc $\varepsilon$ is strongly convex in $\mathcal{F}(\mathrm{P},X)$ when $\mathrm{P}$ is convex and $X$ contains no puncture (points in the interior of~$\mathrm{P}$) and at most one flat vertex (points in the interior of an edge). When $X$ contains at least two punctures or flat vertices, it is also known that this strong convexity property fails. Here, we close the last open case by showing that, for any convex polygon with sufficiently many vertices, one can always place a single puncture in $X$ in such a way that $\mathcal{F}_\varepsilon(\mathrm{P},X)$ is not strongly convex in $\mathcal{F}(\mathrm{P},X)$. We prove a similar result for simple polygons with a single reflex vertex. The main ingredients in our proofs are a decomposition lemma for a class of $3$-dimensional triangulations and a hyperbolic volume argument regarding their embedding into $\mathbb{H}^3$.
\end{abstract}


\section{Introduction}\label{PW2.sec.1}

Consider a simple Euclidean polygon $\mathrm{P}$ and a finite subset $X$ of $\mathrm{P}$ that contains all the vertices of $\mathrm{P}$. The points in $X$ that are not vertices of $\mathrm{P}$ will be called \emph{punctures} of $P$ when they belong to the interior of $\mathrm{P}$ (in analogy with the non-Euclidean case) and \emph{flat vertices} of $\mathrm{P}$ when they belong to the relative interior of an edge. One can decompose $\mathrm{P}$ into triangles by cutting it along a collection $T$ of pairwise non-crossing line segments between points from $X$. Here, we assume that every point in $X$ is a vertex of at least one line segment from $T$. Such a set $T$ is called a \emph{triangulation} of $\mathrm{P}$ with vertex set $X$ and the line segments it contains will be referred to as its \emph{arcs}. 
The set of the triangulations of $\mathrm{P}$ with vertex set $X$ can be given a geometry by considering the graph $\mathcal{F}(\mathrm{P},X)$ whose vertices are these triangulations and whose edges connect two triangulations that differ by a single arc. This graph is called a \emph{flip-graph} because its edges also correspond to the local operation, called a \emph{flip}, that exchanges the diagonals of a convex quadrilateral within a triangulation of $\mathrm{P}$.

The geometry of flip-graphs has received a lot of attention due to its importance in various settings in combinatorics \cite{CardinalHoffmannKustersTothWettstein2018,CardinalSteiner2025,ChangDefantFrischberg2025,ClearyMaio2018,Lee1989,Pournin2014,SleatorTarjanThurston1988}, low-dimensional topology \cite{DisarloParlier2019,GultepeLeininger2017,KorkmazPapadopoulos2012,PandaParlierPournin2025,ParlierPetri2018,ParlierPournin2017}, or computer science \cite{CunhaSauSouzaValencia-Pabon2025,Dorfer2026,HurtadoNoy1999,KanjSedgwickXia2017,LubiwPathak2015,Pilz2014,SleatorTarjanThurston1988}. It has been shown in \cite{SleatorTarjanThurston1988} that, when $\mathrm{P}$ is convex and $X$ is its vertex set, the subgraph $\mathcal{F}_\varepsilon(\mathrm{P},X)$ induced in $\mathcal{F}(\mathrm{P},X)$ by the triangulations that contain a given arc $\varepsilon$ is strongly convex in the sense that all the geodesics in $\mathcal{F}(\mathrm{P},X)$ between two vertices of $\mathcal{F}_\varepsilon(\mathrm{P},X)$ remain entirely in $\mathcal{F}_\varepsilon(\mathrm{P},X)$. This property is instrumental when computing distances in $\mathcal{F}(\mathrm{P},X)$ \cite{ClearyStJohn2018,LiXia2025,Pournin2014,SleatorTarjanThurston1988} and is known to fail when $\mathrm{P}$ is non-convex or when $X$ contains sufficiently-many punctures or flat vertices \cite{LubiwPathak2015,Pilz2014,PourninWang2021}. It is shown in \cite{PourninWang2021} that, already when $\mathrm{P}$ is convex and $X$ contains down to two points that can be either flat vertices or punctures sufficiently close to the boundary, $\mathcal{F}_\varepsilon(\mathrm{P},X)$ is sometimes not even weakly convex in the sense that no geodesic between two vertices of $\mathcal{F}_\varepsilon(\mathrm{P},X)$ is entirely contained in $\mathcal{F}_\varepsilon(\mathrm{P},X)$. When $X$ contains a single flat vertex and no puncture or no flat vertex and a single puncture that is close enough to the boundary, $\mathcal{F}_\varepsilon(\mathrm{P},X)$ is always a strongly convex subgraph of $\mathcal{F}(\mathrm{P},X)$ \cite{ParlierPournin2018b,PourninWang2021}. Likewise, if $X$ contains no flat vertex and a single puncture, that can be placed anywhere in the interior of $\mathrm{P}$, this strong convexity property holds when $\varepsilon$ is incident to the puncture \cite[Theorem 4.3]{ParlierPournin2018b}. Here, we settle the last open case regarding the strong convexity of $\mathcal{F}_\varepsilon(\mathrm{P},X)$ within $\mathcal{F}(\mathrm{P},X)$. 

\begin{thm}\label{PW2.sec.1.thm.0}
Consider a convex polygon $\mathrm{P}$ with vertex set $X$. If $X$ is large enough, then one can place a puncture $p$ in $\mathrm{P}$ such that $\mathcal{F}_\varepsilon(\mathrm{P},X\cup\{p\})$ is not a strongly convex subgraph of $\mathcal{F}(\mathrm{P},X\cup\{p\})$ for some line segment $\varepsilon$.
\end{thm}

We also provide a sufficient condition for this strong convexity property to fail for a prescribed punctured convex polygon. We say that a simple polygon $\mathrm{P}$ is \emph{balanced} with respect to a point $x$ when the number of vertices of $\mathrm{P}$ in any open half space bounded by a line through $x$ is never greater than $3/2$ times the smallest number of vertices in any such half-plane.

\begin{thm}\label{PW2.sec.1.thm.1}
Consider a convex polygon $\mathrm{P}$ with vertex set $X$. Assume that a puncture $p$ has been placed in $\mathrm{P}$ in such a way that $\mathrm{P}$ is balanced with respect to $p$. If in addition, $X$ is large enough, then $\mathcal{F}_\varepsilon(\mathrm{P},X\cup\{p\})$ is not a strongly convex subgraph of $\mathcal{F}(\mathrm{P},X\cup\{p\})$ for some line segment $\varepsilon$.
\end{thm}



Theorem~\ref{PW2.sec.1.thm.1} will be established as a consequence of the following similar statement in the case of a simple polygon with a unique reflex vertex. 

\begin{thm}\label{PW2.sec.1.thm.2}
Consider a simple non-convex polygon $\mathrm{P}$ with vertex set $X$. Assume that $\mathrm{P}$ admits a single reflex vertex and that it is balanced with respect to that vertex. If in addition, $X$ is large enough, then $\mathcal{F}_\varepsilon(\mathrm{P},X)$ is not a strongly convex subgraph of $\mathcal{F}(\mathrm{P},X)$ for some line segment $\varepsilon$.
\end{thm}

\begin{figure}[b]
\begin{centering}
\includegraphics[scale=1]{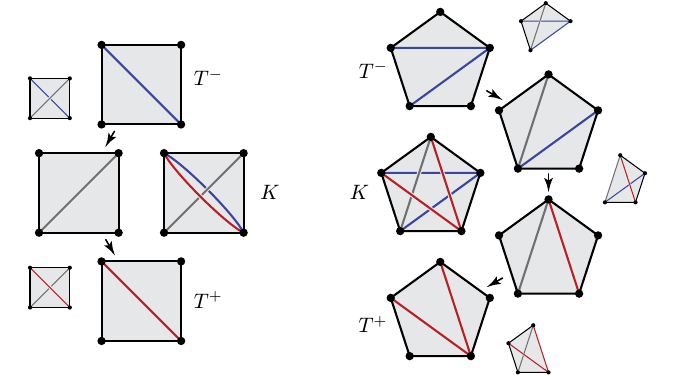}
\caption{Two sequences of flips from $T^-$ to $T^+$, the tetrahedra arising from each flip, and the corresponding blow-up triangulations $K$. On the left, some arcs are bent for clarity.}\label{PW2.sec.1.fig.1}
\end{centering}
\end{figure}

Our proof of Theorems \ref{PW2.sec.1.thm.0}, \ref{PW2.sec.1.thm.1}, and \ref{PW2.sec.1.thm.2} has two main ingredients, both based on the notion of a \emph{blow-up triangulation} that we illustrate in Fig.~\ref{PW2.sec.1.fig.1}. Consider two triangulations $T^-$ and $T ^+$ with the same vertex set $X$, of a simple polygon $\mathrm{P}$ and whose common arcs are contained in the boundary of $\mathrm{P}$. Gluing $T^-$ and $T^+$ along their boundary results in a simplicial triangulated topological sphere and, as observed in~\cite{SleatorTarjanThurston1988}, a path from $T^-$ to $T^+$ in $\mathcal{F}(\mathrm{P},X)$ gives rise to a decomposition $K$ into tetrahedra of the topological ball bounded by that sphere. Each tetrahedron of $K$ corresponds to one of the flips in the considered path. Such a tetrahedral decomposition that arises from a path in $\mathcal{F}(\mathrm{P},X)$ is what we refer to as a blow-up triangulation of $\mathrm{P}$ following the terminology of~\cite{PourninWang2021}. 
It should be noted that $K$ is not always a simplicial complex as it may contain distinct arcs with the same two vertices or tetrahedra glued along a pair of triangular faces (for instance when a flip in the considered path is immediately followed by the inverse flip). However, the strong convexity result from~\cite{SleatorTarjanThurston1988} implies that, if $\mathrm{P}$ is a convex polygon that admits $X$ as its vertex set and $K$ is built from a geodesic path in $\mathcal{F}(\mathrm{P},X)$, then $K$ is a simplicial complex. In fact, it is further shown in \cite{PourninWang2021} that this simplicial complex is flag. This is one of the results from \cite{PourninWang2021} that are called \emph{decomposition lemmas} as they provide triangles along which a blow-up triangulation can be cut into smaller ones, thus making it possible to compute distances in $\mathcal{F}(\mathrm{P},X)$ between triangulation pairs for which that computation was previously out of reach.

The first ingredient in the proof of Theorems \ref{PW2.sec.1.thm.0}, \ref{PW2.sec.1.thm.1}, and \ref{PW2.sec.1.thm.2} is an extension of the decomposition lemmas from~\cite{PourninWang2021} under the least possible restrictive hypotheses to when $\mathrm{P}$ is a simple (non-necessarily convex) polygon and $X$ may contain flat vertices. This extension will not require any new technology beyond what is done in~\cite{PourninWang2021} and we state it as Lemma \ref{PW2.sec.3.lem.5} in Section \ref{PW2.sec.3}. The second ingredient consists in embedding a family of triangulated topological spheres as ideal polytopes in the $3$\nobreakdash-dimensional hyperbolic space $\mathbb{H}^3$ and then lower bounding the number of ideal tetrahedra required to triangulate these polytopes. The lower bound is obtained via the hyperbolic volume argument pioneered (to the best of our knowledge) in \cite{SleatorTarjanThurston1988} and later used in \cite{Smith2000}.

We recall what blow-up triangulations are in Section \ref{PW2.sec.2} and state our generalized decomposition lemma in Section \ref{PW2.sec.3}. The proofs of Theorems \ref{PW2.sec.1.thm.0}, \ref{PW2.sec.1.thm.1}, and~\ref{PW2.sec.1.thm.2} span Sections \ref{PW2.sec.4} to \ref{PW2.sec.6}. In Section \ref{PW2.sec.4}, we explain how these three theorems can be reduced to just one statement and we describe the families of triangulation pairs that we use in Sections \ref{PW2.sec.5} and \ref{PW2.sec.6} to prove that statement. In Section \ref{PW2.sec.5} we apply our decomposition lemma to blow-up triangulations that arise from geodesic paths between these triangulation pairs in order to cut them in smaller pieces. Finally, in Section \ref{PW2.sec.6}, we embed one of these pieces into the $3$-dimensional hyperbolic space $\mathbb{H}^3$ and use a hyperbolic volume argument to complete the proof of Theorems \ref{PW2.sec.1.thm.0}, \ref{PW2.sec.1.thm.1}, and~\ref{PW2.sec.1.thm.2}.

\section{Blow-up triangulations of a simple polygon}\label{PW2.sec.2}

By a simple polygon, we mean a compact subset $\mathrm{P}$ of $\mathbb{R}^2$ whose topological boundary is a Jordan curve obtained as the union of finitely-many line segments with pairwise disjoint relative interiors. We refer to these line segments as the edges of $\mathrm{P}$ and to their endpoints as the vertices of $\mathrm{P}$. We say that two line segments are \emph{crossing} when their relative interiors are non-disjoint. Given a finite subset $X$ of $\mathrm{P}$ that contains all the vertices of $\mathrm{P}$, a triangulation of $\mathrm{P}$ with vertex set $X$ is an inclusion-wise maximal set $T$ of pairwise non-crossing line segments whose two endpoints belong to $X$. We shall refer to the line segments in $T$ as its \emph{boundary arcs} when they are contained in the boundary of $\mathrm{P}$ and as its \emph{interior arcs} otherwise. Note that by the maximality of $T$, the boundary of $\mathrm{P}$ is the union of the boundary arcs of $T$. Cutting $\mathrm{P}$ along the arcs of $T$ results in a collection of triangles which we think of as the triangles of $T$. Two triangulations of a simple polygon are depicted in Fig. \ref{PW2.sec.2.fig.1}.

\begin{figure}[b]
\begin{centering}
\includegraphics[scale=1]{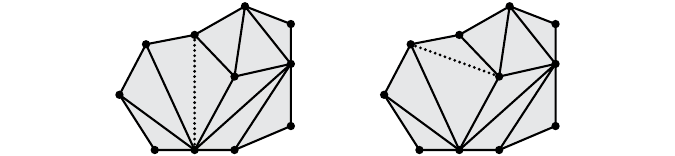}
\caption{Two triangulations of an octagon with one reflex vertex, two flat vertices, and one puncture.}\label{PW2.sec.2.fig.1}
\end{centering}
\end{figure}

Denote by $\mathcal{F}(\mathrm{P},X)$ the graph whose vertices are the triangulations of $\mathrm{P}$ with vertex set $X$ and whose edges connect two triangulations that differ by a single arc. In other words, when two triangulations $T$ and $T'$ are the endpoints of an edge of $\mathcal{F}(\mathrm{P},X)$, there exists an interior arc $\varepsilon$ of $T$ such that gluing the two triangles of $T$ incident to $\varepsilon$ results in a convex quadrilateral and replacing $\varepsilon$ in $T$ by the other diagonal of that quadrilateral produces $T'$. That operation is called a \emph{flip}. For instance, the two triangulations in Fig.~\ref{PW2.sec.2.fig.1} are related the flip that exchanges the dotted arcs. In the remainder of the article, we think of any path in $\mathcal{F}(\mathrm{P},X)$ as a sequence $(T_i)_{0\leq{i}\leq{k}}$ of triangulations of $\mathrm{P}$ with vertex set $X$ such that $T_{i-1}$ can be transformed into $T_i$ by a flip for all $i$.

Now let us now consider such a path $(T_i)_{0\leq{i}\leq{k}}$ in $\mathcal{F}(\mathrm{P},X)$ and explain how that path can be turned into a three dimensional object following the ideas of~\cite{SleatorTarjanThurston1988}. From now on, $\mathbb{R}^2$ will be identified with the plane spanned by the first two coordinates of $\mathbb{R}^3$ and $\pi$ will denote the orthogonal projection from $\mathbb{R}^3$ on $\mathbb{R}^2$. The third coordinate of a point $x$ from $\mathbb{R}^3$ will be denoted by $x_3$. Let us first define the following relation between subsets of $\mathbb{R}^3$.

\begin{defn}\label{PW2.sec.2.defn.1}
Consider two subsets $\sigma$ and $\tau$ of $\mathbb{R}^3$ whose images by $\pi$ are non-disjoint. We say that $\sigma$ is \emph{below} $\tau$ (and that $\tau$ is \emph{above} $\sigma$) when, for any pair of points $x$ in $\sigma$ and $y$ in $\tau$ such that $\pi(x)$ and $\pi(y)$ coincide, $x_3$ is at most $y_3$ and for at least one such pair of points, the inequality is strict.
\end{defn}

Consider a topological disk $\Pi_0$ embedded in $\mathbb{R}^3$ such that $\pi$ induces an homeomorphism $\Pi_0\rightarrow\mathrm{P}$ and the boundary of $\Pi_0$ coincides with that of $\mathrm{P}$. In other words, $\Pi_0$ is obtained as a continuous vertical lift of the interior of $\mathrm{P}$. Now, assuming that $k$ is at least $1$, consider the quadrilateral $\Theta_1$ in $\mathbb{R}^2$ whose diagonals are exchanged by the flip that transforms $T_0$ into $T_1$. Denote by $\Theta_1^-$ the portion of $\Pi_0$ whose image by $\pi$ is $\Theta_1$. Topologically, $\Theta_1^-$ is a disk and in order to model the flip between $T_0$ and $T_1$ we will glue another disk $\Theta_1^+$ along the boundary of $\Theta_1^-$ whose interior lies above that of $\Theta_1^-$. More precisely $\Theta_1^+$ is a disk such that $\pi$ induces an homeomorphism $\Theta_1^+\rightarrow\Theta_1$ while the relative boundaries of $\Theta_1^+$ and $\Theta_1^-$ coincide and for any point $x$ in the relative interior of $\Theta_1^+$, the third coordinate of $x$ is greater than the third coordinate of the point in $\Theta_1^-$ whose image by $\pi$ coincides with that of $x$. By construction,
$$
\Pi_1=\bigl[\Pi_0\mathord{\setminus}\Theta_1^-\bigr]\cup\Theta_1^+
$$
is a topological disk such that $\pi$ induces an homeomorphism $\Pi_1\rightarrow\mathrm{P}$. Moreover, $\Theta_1^-\cup\Theta_1^+$ is a topological sphere. We shall denote by $\psi(1)$ the closed topological ball bounded by that sphere. Repeating this construction inductively from $\Pi_1$ and $T_1$ instead of $\Pi_0$ and $T_0$ allows to build three sequences of topological disks $(\Pi_i)_{0\leq{i}\leq{k}}$ and $(\Theta_i^\pm)_{0\leq{i}\leq{k}}$ such that $\pi$ induces homeomorphisms $\Pi_i\rightarrow\mathrm{P}$ and $\Theta_i^\pm\rightarrow\Theta_i$ where $\Theta_i$ denotes the quadrilateral whose diagonals are exchanged by the flip between $T_{i-1}$ and $T_i$. Moreover $\Theta_i^-$ is below $\Theta_i ^+$ in the sense of Definition~\ref{PW2.sec.2.defn.1} and the union of $\Theta_i^-$ with $\Theta_i^+$ is a topological sphere that bounds a closed topological ball that we will denote by $\psi(i)$.

We can decompose $\Pi_i$ in the same way that $T_i$ decomposes $\mathrm{P}$ using the disks, arcs, and points in $\Pi_i$ whose images by $\pi$ are the triangles, edges, and vertices of $T_i$. Consider the set $K$ made of all these disks, arcs and triangles when $i$ ranges from $0$ to $k$, together with the $3$-dimensional balls $\psi(i)$ when $i$ ranges from $1$ to $k$. We will refer to the disks as triangles and to the $3$-dimensional balls as tetrahedra. Indeed, each of these triangles is bounded by three arcs of $K$ and the points, arcs, and triangles in the boundary of each of these tetrahedra form a simplicial complex isomorphic to the boundary of a Euclidean tetrahedron. We will also refer to the elements of $K$ as its \emph{faces}.

As one can see on the left of Fig.~\ref{PW2.sec.1.fig.1}, $K$ is not necessarily a simplicial complex. Still, it decomposes the union of its faces into tetrahedra in a natural way. We will refer to $K$ as a \emph{blow-up triangulation} of $\mathrm{P}$ following the terminology from~\cite{PourninWang2021} as it decomposes a portion of $\mathbb{R}^3$ obtained by blowing up $\mathrm{P}$ vertically. While there are infinitely many ways to place the surfaces $\Pi_i$ in $\mathbb{R}^3$, all of them are identical up to the homeomorphisms of $\mathbb{R}^3$ that fix the boundary of $\mathrm{P}$ pointwise while sending every point in $\mathbb{R}^3$ to a point with the same two first coordinates and preserving the order of the third coordinate of distinct points with the same two first coordinates. We will consider blow-up triangulations of $\mathrm{P}$ up to these homeomorphisms and therefore, each path in $\mathcal{F}(\mathrm{P},X)$ will correspond with a unique blow-up triangulation of $\mathrm{P}$. However as noted in \cite{PourninWang2021}, a given blow-up triangulation of $\mathrm{P}$ may correspond to several different paths. We sum up the construction of blow-up triangulations as follows.

\begin{prop}\label{PW2.sec.2.prop.1}
Consider a path $(T_i)_{0\leq{i}\leq{k}}$ in $\mathcal{F}(\mathrm{P},X)$ and the blow-up triangulation $K$ of $\mathrm{P}$ that corresponds to that path. There exists a bijection $\psi$ from $\{1,\ldots,k\}$ to the set of the tetrahedra in $K$ and a family $(\Pi_i)_{0\leq{i}\leq{k}}$ of topological disks embedded within $\mathbb{R}^3$ such that
\begin{enumerate}
\item[(i)] $\pi$ induces homeomorphisms $\Pi_i\rightarrow\mathrm{P}$, 
\item[(ii)] if $i$ is less than $j$, then $\Pi_i$ is below $\psi(j)$ and $\Pi_j$,
\item[(iii)] if $\sigma$ is an arc or a triangle of $T_i$, then $\pi^{-1}(\sigma)\cap\Pi_i$ belongs to $K$, and
\item[(iv)] the image of $\psi(i)$ by $\pi$ is the quadrilateral whose diagonals are exchanged by the flip between $T_{i-1}$ and $T_i$.
\end{enumerate}
\end{prop}

It immediately follows from Proposition \ref{PW2.sec.2.prop.1} that restricting of aboveness relationship from Definition \ref{PW2.sec.2.defn.1} to the faces of a blow-up triangulation and then extending this restriction by transitivity results in a partial order.

\begin{prop}\label{PW2.sec.2.prop.2}
The transitive closure of the aboveness relationship is a partial order on any blow-up triangulation of $\mathrm{P}$.
\end{prop}

\section{A generalized decomposition lemma}\label{PW2.sec.3}

In this section, we consider a path $(T_i)_{0\leq{i}\leq{k}}$ in $\mathcal{F}(\mathrm{P},X)$ and the associated blow-up triangulation $K$ of $\mathrm{P}$,  where $X$ does not contain any puncture of $\mathrm{P}$. We will denote by $\Pi_0$ to $\Pi_k$ the topological disks provided by Proposition \ref{PW2.sec.2.prop.1}, that are obtained when building $K$ from $(T_i)_{0\leq{i}\leq{k}}$. In particular $\pi$ induces homeomorphisms $\Pi_i\rightarrow\mathrm{P}$. We also denote by $\psi(i)$ the tetrahedron in $K$ that corresponds to the flip between $T_{i-1}$ and $T_i$.

The decomposition lemmas and the flagness property established in \cite{PourninWang2021} deal with a topological circle $C$ obtained as the union of three arcs $\alpha$, $\beta$, and $\gamma$ of $K$. More precisely, they state that under certain conditions, $K$ must contain a triangle bounded by $C$ or whose image by $\pi$ admits $\pi(C)$ as its boundary. In order to state our generalization of these decomposition lemmas, we first recall some terminology from \cite{PourninWang2021}. From now on $\alpha$, $\beta$, $\gamma$, and $C$ are fixed. We also assume that $\pi(\alpha)$ is the arc removed by the flip that transform $T_0$ into $T_1$ and that $\pi(\gamma)$ is the arc introduced by the flip that transforms $T_{k-1}$ into $T_k$. We will see when proving our generalized decomposition lemma that this assumption can be relaxed, but it will allow for a simpler exposition. We will denote by $a$, $b$ and $c$ the points of $X$ contained in $C$ with the convention that $a$ does not belong to $\alpha$, $b$ does not belong to $\beta$, and $c$ does not belong to $\gamma$.

\begin{defn}\label{PW2.sec.3.defn.1}
We say that an arc $\varepsilon$ of $K$ \emph{penetrates} $C$ when $\varepsilon$ is both below and above $C$ or, equivalently when it is below one of the arcs $\alpha$, $\beta$, or $\gamma$ and above another one of these arcs. We also say that a tetrahedron of $K$ penetrates $C$ when it is incident to an arc of $K$ that penetrates $C$. 
\end{defn}

The set of all the tetrahedra of $K$ that penetrate $C$ will be denoted by $\mathcal{P}$ in the sequel. We shall consider two other subsets of the tetrahedra of $K$ in order to state our generalized decomposition lemma. Observe that, when $\pi(\alpha)$, $\pi(\beta)$, and $\pi(\gamma)$ do not belong to the boundary of $\mathrm{P}$, cutting that polygon along these arcs results in the triangle bounded by $\pi(C)$ and in three additional simple polygons. We will denote by $\mathrm{P}^\alpha$, $\mathrm{P}^\beta$, and $\mathrm{P}^\gamma$ these three polygons with the convention that $\mathrm{P}^\varepsilon$ is incident to $\pi(\varepsilon)$ as shown in Fig. \ref{PW2.sec.3.fig.1}. When $\pi(\varepsilon)$ is contained in the boundary of $\mathrm{P}$, we will think of $\mathrm{P}^\varepsilon$ as just the arc $\pi(\varepsilon)$ and whenever we need a triangulation of $\mathrm{P}^\varepsilon$ in that case, we mean the singleton $\{\varepsilon\}$ (which is consistent with what a $1$-dimensional triangulation is).

\begin{figure}
\begin{centering}
\includegraphics[scale=1]{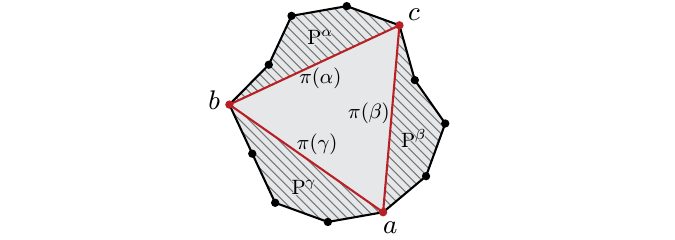}
\caption{$\mathrm{P}^\alpha$, $\mathrm{P}^\beta$, and $\mathrm{P}^\gamma$ (striped).}\label{PW2.sec.3.fig.1}
\end{centering}
\end{figure}

Let $q$ be an index such that $\beta$ is contained in $\Pi_q$. There may be several such indices but we fix one of them from now on. Denote
$$
\Pi^\varepsilon=\pi^{-1}(\mathrm{P}^\varepsilon)\cap\Pi_i
$$
where $\varepsilon$ is equal to $\alpha$, $\beta$, or $\gamma$, and the index $i$ is taken equal to  $0$, to $q$, or to $k$ depending on whether $\varepsilon$ is equal to $\alpha$, $\beta$, or $\gamma$, respectively.

We shall now consider two additional subsets of the tetrahedra of $K$.

\begin{defn}\label{PW2.sec.3.defn.2}
A tetrahedron $\sigma$ from $K$ is called a \emph{lower tetrahedron} of $K$ when it is incident to a triangle $\tau$ contained $K$ such that
\begin{itemize}
\item[(i)] $\tau$ is not above $\Pi^\beta$,
\item[(ii)] $\tau$ does not have a vertex in $\mathrm{P}^\alpha\mathord{\setminus}\{b\}$, and
\item[(iii)] if $\tau$ is contained in $\Pi^\beta$, then $\Pi^\beta$ is above $\sigma$,
\end{itemize}
and an \emph{upper tetrahedron} of $K$ when $\sigma$ is incident to a triangle $\tau$ of $K$ such that (i), (ii), and (iii) hold when ``above'' is replaced by ``below'' and $\alpha$ by $\gamma$.
\end{defn}

Let us denote by $\mathcal{L}$ and $\mathcal{U}$ the sets of the lower and upper tetrahedra of $K$. The following is proven in \cite{PourninWang2021} (see Lemma 5.7 therein) in the case when $\mathrm{P}$ is convex and $X$ is its vertex set but as mentioned at the end of \cite[Section 5]{PourninWang2021}, the proof still holds when $\mathrm{P}$ is non-convex and $X$ contains flat vertices.

\begin{lem}\label{PW2.sec.3.lem.2}
$\mathcal{L}$ and $\mathcal{U}$ are disjoint.
\end{lem}

Note that Definition \ref{PW2.sec.3.defn.2} depends not only on $q$ but also on $\Pi^\beta$ and therefore on which path corresponding to $K$ in $\mathcal{F}(\mathrm{P},X)$ has been chosen. However, the intersections $\mathcal{P}\cap\mathcal{L}$ and $\mathcal{P}\cap\mathcal{U}$ really only depend on $\alpha$, $\beta$, and $\gamma$.

\begin{prop}\label{PW2.sec.3.prop.3}
A tetrahedron $\sigma$ from $\mathcal{P}$ is a lower tetrahedron of $K$ if and only if it is incident to a unique triangle $\tau$ contained $K$ such that
\begin{itemize}
\item[(i)] $\tau$ is not above $\beta$,
\item[(ii)] $\tau$ does not have a vertex in $\mathrm{P}^\alpha\mathord{\setminus}\{b\}$, and
\item[(iii)] if $\sigma$ is above $\beta$, then $\tau$ has its three vertices in $\mathrm{P}^\gamma$,
\end{itemize}
and $\sigma$ is an upper tetrahedron of $K$ if and only if \emph{(i)}, \emph{(ii)}, and \emph{(iii)} hold when ``above'' is replaced by ``below'' while $\alpha$ and $\gamma$ are exchanged.
\end{prop}
\begin{proof}
Consider a tetrahedron $\sigma$ from $\mathcal{P}$. Assume that some triangle $\tau$ in $K$ incident to $\sigma$ is not above $\beta$ and does not have a vertex in $\mathrm{P}^\alpha\mathord{\setminus}\{b\}$. If $\tau$ has its three vertices in $\mathrm{P}^\gamma$ then it cannot be above or in $\Pi^\beta$ and by Definition~\ref{PW2.sec.3.defn.2}, $\sigma$ belongs to $\mathcal{L}$. Now assume that $\sigma$ is not above $\beta$ but $\tau$ (and therefore $\sigma$) has a vertex in $\mathrm{P}^\beta\mathord{\setminus}\pi(\beta)$. As in addition, $\sigma$ belongs to $\mathcal{P}$, then this tetrahedron must be above $\alpha$ and as a consequence, some arc in $K$ incident to $\sigma$ must be above $\alpha$ and below $\beta$. Hence, $\tau$ cannot be above $\Pi^\beta$ and if it is contained in $\Pi^\beta$, then it must be above $\sigma$. By Definition~\ref{PW2.sec.3.defn.2}, $\sigma$ belongs to $\mathcal{L}$ as well.

Now assume that $\sigma$ belongs to $\mathcal{L}$ and consider a triangle $\tau$ in $K$ incident to $\sigma$ that satisfies the assertions (i), (ii), and (iii) in the statement of Definition~\ref{PW2.sec.3.defn.2}. As $\tau$ is not above $\Pi^\beta$, it cannot be above $\beta$. Further assume that $\sigma$ is above $\beta$. In that case, $\tau$ cannot be below $\Pi^\beta$ because $\sigma$ cannot be both above and below $\Pi^\beta$. Likewise $\tau$ cannot be contained in $\Pi^\beta$ because in that case, the assrtion (iii) from the statement of Definition \ref{PW2.sec.3.defn.2} would imply that $\sigma$ is below $\Pi^\beta$. As a consequence, $\tau$ is not above, below, or contained in $\Pi^\beta$. In particular that triangle cannot have a vertex in $\Pi^\beta\mathord{\setminus}\pi(\beta)$. As it does not have a vertex in $\mathrm{P}^\alpha\mathord{\setminus}\{b\}$ either, then its three vertices are contained in $\mathrm{P}^\gamma$.

Finally observe that, if $\sigma$ is not above $\beta$, then as it belongs to $\mathcal{P}$, it must be above $\alpha$ and therefore, have a vertex in $\mathrm{P}^\alpha\mathord{\setminus}\pi(\alpha)$. Hence, no other triangle than $\tau$ in $K$ is incident to $\sigma$ and has no vertex in $\mathrm{P}^\alpha\mathord{\setminus}\{b\}$. If however $\sigma$ is above $\beta$, then is must have a vertex in $\mathrm{P}^\beta\mathord{\setminus}\pi(\beta)$. In that case, $\tau$ has its three vertices in $\mathrm{P}^\gamma$ and there connot be another such triangle in $K$ incident to $\sigma$. One shows the proposition when $\sigma$ is an upper tetrahedra of $K$ by exchanging $\alpha$ with $\gamma$, reversing the aboveness relation, and replicating this proof.
\end{proof}

Let us recall another of the properties that were established in \cite{PourninWang2021} when $\mathrm{P}$ is a convex polygon but that still holds in the more general case of a simple polygon. Given a arc $\varepsilon$ of $K$, we denote by $S_K(\varepsilon)$ the \emph{star} of $\varepsilon$ in $K$:
$$
S_K(\varepsilon)=\{\sigma\in{K}:\varepsilon\subset\sigma\}\mbox{.}
$$

The following statement is proven in \cite{PourninWang2021} (see Lemma  3.2 therein) when $\mathrm{P}$ is convex and $X$ is the vertex of $\mathrm{P}$. However, its proof does not require either of these hypotheses and can be repeated word for word when $\mathrm{P}$ is simple but not necessarily convex and $X$ possibly contains flat vertices of $\mathrm{P}$.

\begin{lem}\label{PW2.sec.3.lem.3}
If $S_K(\beta)$ does not contain a triangle $\tau$ such that $\pi(C)$ is the boundary of $\pi(\tau)$, then some tetrahedron in $S_K(\beta)$ penetrates $C$. 
\end{lem}

\begin{figure}[b]
\begin{centering}
\includegraphics[scale=1]{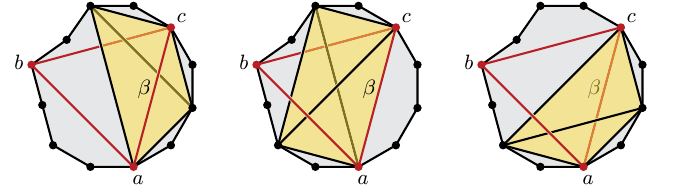}
\caption{The possible tetrahedra in $S_K(\beta)$ that penetrate $C$.}\label{PW2.sec.3.fig.2}
\end{centering}
\end{figure}

Theorem 5.8 from \cite{PourninWang2021} states that if $\mathrm{P}$ is convex and $X$ is its vertex set then there exists a blow-up triangulation $N$ of $\mathrm{P}$ that corresponds to a path of length $|\mathcal{L}\cup\mathcal{U}|$ in $\mathcal{F}(\mathrm{P},X)$ from $T_0$ to $T_k$. Assuming that the distance of $T_0$ and $T_k$ in $\mathcal{F}(\mathrm{P},X)$ is equal to $k$ or equivalently to the number of tetrahedra in $K$, it then follows from Lemma \ref{PW2.sec.3.lem.2} that the set of the tetrahedra of $K$ is the disjoint union of $\mathcal{L}$ and $\mathcal{U}$. According to Lemma \ref{PW2.sec.3.lem.3}, this can only happen when the image by $\pi$ of some triangle in $S_K(\beta)$ admits $\pi(C)$ as its boundary. Otherwise that lemma states that $S_K(\beta)$ contains one of the tetrahedra shown in Fig. \ref{PW2.sec.3.fig.2}. However, by Proposition \ref{PW2.sec.3.prop.3}, none of them can belong to $\mathcal{L}$ or $\mathcal{U}$. Indeed, as one can see in the figure, every triangular face of these tetrahedra is above $\beta$ or has a vertex in $\mathrm{P}^\alpha\mathord{\setminus}\{b\}$ and below $\beta$ or has a vertex in $\mathrm{P}^\gamma\mathord{\setminus}\{b\}$. This contradicts the above statement that the set of the tetrahedra of $K$ is the disjoint union of $\mathcal{L}$ and $\mathcal{U}$. Hence, we obtain the following decomposition lemma.

\begin{lem}[{\cite[Lemma 5.10]{PourninWang2021}}]\label{PW2.sec.3.lem.4}
Assume that $\mathrm{P}$ is convex and that $X$ is the vertex set of $\mathrm{P}$. If $(T_i)_{0\leq{i}\leq{k}}$ is a geodesic path in $\mathcal{F}(\mathrm{P},X)$, then $S_K(\beta)$ must contain a triangle $\tau$ such that $\pi(\tau)$ admits $\pi(C)$ for its boundary. 
\end{lem}

The only obstacle for Lemma \ref{PW2.sec.3.lem.4} to hold in the case when $\mathrm{P}$ is non-convex and $X$ may contain flat vertices is the existence of the blow-up triangulation $N$ provided by Theorem 5.8 from \cite{PourninWang2021}. While $N$ may indeed not exist in this more general case, we will provide sufficient condition on $K$ so that it does. In order to do that, let us recall how $N$ is built from $K$. According to Proposition \ref{PW2.sec.3.prop.3}, a lower tetrahedron $\sigma$ of $K$ that penetrates $C$ is incident to a unique triangle $\tau$ in $K$ that satisfies the three assertions (i), (ii), and (iii) in its statement. Likewise if $\sigma$ is an upper tetrahedron of $K$ that penetrates $C$, there is a unique triangle $\tau$ of $K$ incident to $\sigma$ that satisfies these assertions where ``above'' is replaced by ``below'' and $\alpha$ by $\gamma$. In both cases, we refer to $\tau$ as the \emph{base} of $\sigma$. 

\begin{defn}\label{PW2.sec.3.defn.3}
Consider a tetrahedron $\sigma$ in $\mathcal{P}\cap(\mathcal{L}\cup\mathcal{U})$ with base $\tau$ and a point $x$ in $X$ that is not a vertex of $\sigma$. We say that $\sigma$ \emph{can be pulled to $x$} when the quadrilateral obtained as the convex hull of $\pi(\tau)\cup\{x\}$ is contained in $\mathrm{P}$ and does not contain any point of $X$ other than its vertices.
\end{defn}

When a tetrahedron $\sigma$ in $\mathcal{P}\cap(\mathcal{L}\cup\mathcal{U})$ with base $\tau$ can be pulled to a point $x$ in $X$ that is not a vertex of $\sigma$, the pulling operation results in a tetrahedron $\sigma'$ that is still incident to $\tau$ but whose fourth vertex is $x$. Moreover, $\tau$ is below $\sigma'$ if and only if it is below $\sigma$. In the case when $\mathrm{P}$ is convex and $X$ is its vertex set, the blow-up triangulation $N$ provided by Theorem~5.8 from \cite{PourninWang2021} is built from $\mathcal{L}\cup\mathcal{U}$ by pulling all the tetrahedra in $\mathcal{P}\cap\mathcal{L}$ to $c$, by pulling all the tetrahedra in $\mathcal{P}\cap\mathcal{U}$ to $a$, and by leaving the tetrahedra in $(\mathcal{L}\cup\mathcal{U})\mathord{\setminus}\mathcal{P}$ unaffected. It is observed at the end of \cite[Section 5]{PourninWang2021} that in fact, the existence of $N$ really only depends on whether all of these pulling operations are possible, leading to a generalized decomposition lemma in the case when $\mathrm{P}$ is still convex but $X$ may contain flat vertices (see \cite[Lemma 5.12]{PourninWang2021}). This observation turns out to remain true more generally when $\mathrm{P}$ is non-convex: one can still build $N$ using these pulling operations on the condition that the lower and upper tetrahedra of $K$ that penetrate $C$ can all be pulled to $c$ and $a$, respectively, in the sense of Definition \ref{PW2.sec.3.defn.3}. We therefore immediately obtain the following.

\begin{lem}\label{PW2.sec.3.lem.5}
Consider a simple polygon $\mathrm{P}$ and a finite subset $X$ of the boundary of $\mathrm{P}$ that contains all of the vertices of $\mathrm{P}$. Further consider a blow-up triangulation $K$ of $\mathrm{P}$ that corresponds to a  geodesic path in $\mathcal{F}(\mathrm{P},X)$ and three arcs $\alpha$, $\beta$, and $\gamma$ of $K$ whose union is a topological circle $C$. If
\begin{itemize}
\item[(i)] $K$ does not have any face below $\alpha$ or above $\gamma$,
\item[(ii)] every tetrahedron in $\mathcal{P}\cap\mathcal{L}$ can be pulled to $c$, and
\item[(iii)] every tetrahedron in $\mathcal{P}\cap\mathcal{U}$ can be pulled to $a$,
\end{itemize}
then $\pi(C)$ is the boundary of the image by $\pi$ of some triangle in $S_K(\beta)$. 
\end{lem}

\begin{rem}\label{PW2.sec.3.rem.1}
The requirement that $K$ does not have any face below $\alpha$ or above $\gamma$ can be relaxed in the statement of Lemma \ref{PW2.sec.3.lem.5} into only asking that $\alpha$ is not above $\beta$ or $\gamma$ and that $\gamma$ is not below $\alpha$ or $\beta$ with respect to the transitive closure  of the aboveness relation. Indeed, under this weaker assumption, observe that according to Proposition~\ref{PW2.sec.2.prop.2} one can remove all the faces of $K$ below $\alpha$ or above $\gamma$, resulting in a smaller blow-up triangulation of $\mathrm{P}$ that still contains $\alpha$, $\beta$, and $\gamma$ but that satisfies the statement of Lemma \ref{PW2.sec.3.lem.5}.
\end{rem}

\section{Large polygons with a single reflex vertex}\label{PW2.sec.4}

The main aim of this section is to describe the family of triangulation pairs that the proofs of Theorems \ref{PW2.sec.1.thm.0}, \ref{PW2.sec.1.thm.1}, and \ref{PW2.sec.1.thm.2} will be based on. We will also show how these three proofs can be reduced to establishing a single statement. For the remainder of the article, $n$ and $m$ denote two fixed integers such that $n$ is odd and at least $3$ while $m$ is non-negative. Denote by $H$ and $V$ the horizontal and vertical lines of $\mathbb{R}^2$ through the origin. The announced triangulations will be triangulations of polygons contained in the following set $\mathcal{Q}$.

\begin{figure}[b]
\begin{centering}
\includegraphics[scale=1]{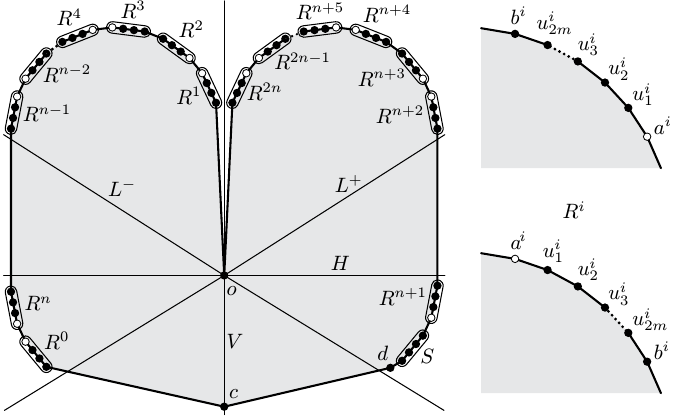}
\caption{The polygon $\mathrm{Q}$ (left) and the labeling of $R^i$ when $i$ is even (top right) and odd (bottom right).}\label{PW2.sec.4.fig.1}
\end{centering}
\end{figure}

\begin{defn}\label{PW2.sec.4.defn.1}
A polygon $\mathrm{Q}$ belongs to $\mathcal{Q}$ when the origin of $\mathbb{R}^2$ is its unique reflex vertex and there exist two lines $L^-$ and $L^+$ such that
\begin{itemize}
\item[(i)] the only vertex of $\mathrm{Q}$ in $H$, $L^-$, and $L^+$ is the origin of $\mathbb{R}^2$,
\item[(ii)] $V\cap\mathrm{Q}$ is the line segment one extremity of which is the origin of $\mathbb{R}^2$ and the other a vertex $c$ of $\mathrm{Q}$ contained in the lower half-plane,
\item[(iii)] the two portions of $\mathrm{Q}$ in the upper half-plane above $L^-$ and $L^+$ and on each side of $V$ contain exactly $2(n-1)(m+1)$ vertices of $\mathrm{Q}$,
\item[(iv)] the portion of $\mathrm{Q}$ above $L^-$ in the lower half-plane contains at least $8n(m+1)+6m+9$ vertices of $\mathrm{Q}$ and the portion above $L^+$ in the lower half-plane contains exactly $4(m+1)$ vertices of $\mathrm{Q}$, and
\item[(v)] there is no vertex of $\mathrm{Q}$ below $L^-$ or $L^+$ in the upper half-plane and the only vertex of $\mathrm{Q}$ below $L^-$ and $L^+$ in the lower half-plane is $c$.
\end{itemize}
\end{defn}

Note that $\mathcal{Q}$ depends on $n$ and $m$ but we will not mark this dependence in order to not to overburden the notation. A polygon from $\mathcal{Q}$ is sketched in Fig. \ref{PW2.sec.4.fig.1}. For readability, apart from $o$, $c$, and $d$, the vertices of $\mathrm{Q}$ are grouped in the figure into $2n+2$ subsets of consecutive vertices that we denote by $S$ and $R^0$ to $R^{2n}$. Each of the $R^i$ contains $2(m+1)$ vertices while $S$ contains at least $8n(m+1)+4m+6$ vertices. The labeling of the vertices of $R^i$ that we shall use throughout the article is indicated on the right of Fig. \ref{PW2.sec.4.fig.1}: if $i$ is even, then the vertices in $R^i$ are labeled by $a^i$, then $u^i_1$ to $u^i_{2m}$, then $b^i$ counter-clockwise around $\mathrm{Q}$. If $i$ is odd, then the labeling is the same except that it is done clockwise around $\mathrm{Q}$ instead of counter-clockwise as shown bottom right in Fig.~\ref{PW2.sec.4.fig.1}. The vertices $a^0$ to $a^{2n}$ are marked using outlined dots.

The purpose of the next two sections is to prove the following theorem.

\begin{thm}\label{PW2.sec.4.lem.2}
Consider a polygon $\mathrm{Q}$ in $\mathcal{Q}$ with vertex set $X$. If $n$ and $m/n$ are both large enough, then there exists a line segment $\varepsilon$ such that $\mathcal{F}_\varepsilon(\mathrm{Q},X)$ is not a strongly convex subgraph of $\mathcal{F}(\mathrm{Q},X)$.
\end{thm}

Before we proceed to describing the families of triangulations of the polgyons in $\mathcal{Q}$ that our proof of Theorem \ref{PW2.sec.4.lem.2} will be based on, let us show how Theorems~\ref{PW2.sec.1.thm.0}, \ref{PW2.sec.1.thm.1} and~\ref{PW2.sec.1.thm.2} follow from Theorem \ref{PW2.sec.4.lem.2}.

\begin{lem}\label{PW2.sec.4.lem.1.5}
Consider a convex polygon $\mathrm{P}$ with vertex set $X$. If $X$ is large enough, then there exists a puncture $p$ of $\mathrm{P}$ such that the vertex set of some polygon in $\mathcal{Q}$ is a subset of $X\cup\{p\}$ up to an affine transformation of $\mathbb{R}^2$.
\end{lem}
\begin{proof}
Let us first consider four vertices of $\mathrm{P}$ that split its boundary into portions containing approximately the same number of vertices. We can then assume, by using an affine tranformation of $\mathbb{R}^2$ if needed, that $H$ and $V$ each contain two of these vertices while each of the four open orthants of $\mathbb{R}^2$ contain $|X|/4$ points from $X$ minus a number that is bounded independently from $|X|$. Further consider a line $L^+$ through the origin of $\mathbb{R}^2$ that splits the vertices of $\mathrm{P}$ in the open negative orthant in portions of about the same size. Likewise let $L^-$ be a line that splits the vertices of $\mathrm{P}$ contained in the lower right open orthant into two approximately equal subsets.

\begin{figure}[b]
\begin{centering}
\includegraphics[scale=1]{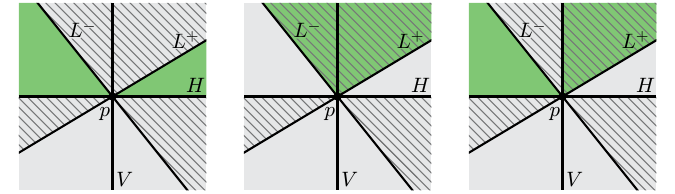}
\caption{$H$, $V$, $L^-$ and $L^+$ divide $\mathbb{R}^2$ into eight open cones. The green cones possibly contain too few vertices of $\mathrm{P}$.}\label{PW2.sec.4.fig.0}
\end{centering}
\end{figure}

Now the situation is as shown on Fig.~\ref{PW2.sec.4.fig.0}: the lines $H$, $V$, $L^-$ and $L^+$ define eight open cones and the four cones that are contained in the lower half-plane each contain at least $|X|/8$ points from $X$ minus a constant $r$ that is independent from $|X|$. Let us show that, up to some affine transformation of $\mathbb{R}^2$, the four cones that are striped in the figure contain at least $|X|/8-r$ points from $X$. The four open cones contained in the lower half-plane already satisfy this requirement and, as there are at least about $|X|/4$ points from $X$ in each of the orthants in the upper half-plane, at least two of the four open cones in the upper half-plane also satisfy it. The only three possibilities for how these cones are distributed up to symmetry are shown in Fig. \ref{PW2.sec.4.fig.0}. In the first situation on the left of the figure, the four striped cones satisfy our requirement. In the second situation in the center of the figure, it suffices to take the symmetric of $\mathbb{R}^2$ with respect to $H$ in order for the desired condition to be met and in the third one, it suffices to rotate $\mathbb{R}^2$ by a quarter turn  counter-clockwise. After these operations, $L^-$ and $L^+$ need to be relabeled appropriately.

Now that the two cones above $L^-$ and $L^+$ in the upper half-plane and the two cones above $L^-$ or $L^+$ in the lower half-plane contain a constant fraction of the vertices of $\mathrm{P}$ then, provided $X$ is large enough, we can pick $2(n-1)(m+1)$ vertices of $\mathrm{P}$ in the former two cones, $8n(m+1)+6m+9$ vertices of $\mathrm{P}$ in the cone above $L^-$ in the lower half-plane, and $4(m+1)$ vertices of $\mathrm{P}$ in the cone above $L^+$ in the lower half-plane. The resulting set of points, together with the vertex $c$ of $\mathrm{P}$ on $V$ in the lower half-plane, and the origin of $\mathrm{R}^2$ (that serves as the puncture $p$) is the vertex set of a polygon contained in $\mathcal{Q}$.
\end{proof}

Together with Theorem \ref{PW2.sec.4.lem.2}, Lemma \ref{PW2.sec.4.lem.1.5} will allow to prove Theorem \ref{PW2.sec.1.thm.0}. The following lemma will serve the same purpose for both Theorem~\ref{PW2.sec.1.thm.1} and Theorem~\ref{PW2.sec.1.thm.2}. In its statement, by a \emph{subpolygon} of $\mathrm{P}$, we mean a simple polygon whose vertex set is a subset of the vertex set of $\mathrm{P}$. 

\begin{lem}\label{PW2.sec.4.lem.1}
Consider a simple polygon $\mathrm{P}$ with vertex set $X$ and a single reflex vertex $o$. If $\mathrm{P}$ is balanced with respect to $o$ and $X$ is large enough, then some polygon in $\mathcal{Q}$ is a subpolygon of $\mathrm{P}$ up to an affine transformation of $\mathbb{R}^2$.
\end{lem}
\begin{proof}
Denote by $\nu$ the smallest number of vertices of $\mathrm{P}$ in an open half-plane bounded by a line through $o$. Assuming that $\mathrm{P}$ is balanced, we have
\begin{equation}\label{PW2.sec.4.lem.1.eq.0}
\nu\geq\frac{2}{5}(|X|-2)\mbox{.}
\end{equation}

For any vertex $x$ of $\mathrm{P}$, denote by $\rho(x)$ the next vertex of $\mathrm{P}$ clockwise. If $k$ is a non-zero integer with absolute value at most $|X\mathord{\setminus}\{o\}|$, observe that $\rho(o)$ and $\rho^{-1}(o)$ must be contained in the same closed half-plane of $\mathbb{R}^2$ bounded by the line through $o$ and $\rho^k(o)$. Let $Y(k)$ be the intersection of $X$ with the open half-plane of $\mathbb{R}^2$ bounded by that line and that contains either $\rho(o)$ or $\rho^{-1}(o)$. Consider the smallest positive integers $k$ and $l$ such that the sets
$$
X\mathord{\setminus}\bigl(Y(k)\cup{Y(-1)}\bigr)
$$
and
$$
X\mathord{\setminus}\bigl(Y(-l)\cup{Y(1)}\bigr)
$$
are both non-empty. We will assume without loss of generality that $k$ is at most $l$ by, if needed, applying a central symmetry with respect to $o$ and we will denote by $c$ one of the points contained in the former set. 

By construction, if $k$ is at least $2$, then
\begin{equation}\label{PW2.sec.4.lem.1.eq.0.7}
X\mathord{\setminus}\bigl(Y(k-1)\cup{Y(1-l)}\bigr)=\emptyset
\end{equation}
as shown on the left of Fig.~\ref{PW2.sec.4.fig.4}. However, it follows from the definition of $\nu$ that there are at least $\nu$ points of $X$ in the set $X\mathord{\setminus}Y(k-1)$. According to (\ref{PW2.sec.4.lem.1.eq.0.7}), all of these points belong to $Y(1-l)$ and as a consequence,
\begin{equation}\label{PW2.sec.4.lem.1.eq.0.5}
\bigl|Y(1-l)\mathord{\setminus}Y(k-1)\bigr|\geq\nu\mbox{.}
\end{equation}

By symmetry, the same holds when exchanging $Y(k-1)$ and $Y(1-l)$. However, observe that the intersection of $Y(k-1)$ with $Y(1-l)$ contains exactly $k+l-4$ points. Therefore, by (\ref{PW2.sec.4.lem.1.eq.0.5}), the number of points contained in $Y(1-l)$ is at least $\nu+k+l-4$ and the same holds for $Y(k-1)$. As $\mathrm{P}$ is balanced with respect to $o$ and $k$ is not greater than $l$ it follows that
\begin{equation}\label{PW2.sec.4.lem.1.eq.1}
k\leq\frac{\nu}{4}+2\mbox{.}
\end{equation}

\begin{figure}[b]
\begin{centering}
\includegraphics[scale=1]{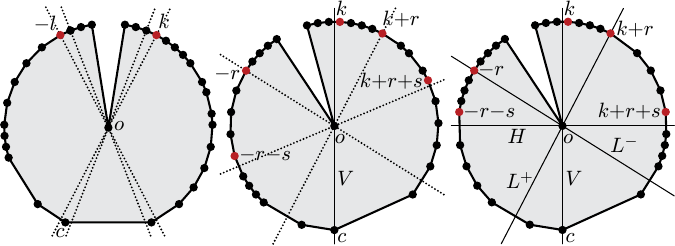}
\caption{The construction in Lemma \ref{PW2.sec.4.lem.1}. Vertices $\rho^i(o)$ are just labeled $i$ in order not to overburden the figure.}\label{PW2.sec.4.fig.4}
\end{centering}
\end{figure}

Note that when $k$ is equal to $1$, then this inequality is immediate.

Let us now assume that $\mathrm{P}$ is placed within $\mathbb{R}^2$ as shown in the center of Fig.~\ref{PW2.sec.4.fig.4} by if needed applying an affine transformation of $\mathbb{R}^2$. In particular, $o$ is placed at the origin of $\mathbb{R}^2$ and $c$ on $V$ in the lower half-plane. Denote
\begin{equation}\label{PW2.sec.4.lem.1.eq.2}
r=\biggl\lfloor\frac{\nu}{9}\biggr\rfloor\mbox{.}
\end{equation}

By (\ref{PW2.sec.4.lem.1.eq.1}) and (\ref{PW2.sec.4.lem.1.eq.2}), $k+2r-2$ is not greater than $\nu$. As a consequence, the intersection of $Y(k+r)$ and $Y(-r)$ is contained in an open half-plane whose boundary contains the origin. As this intersection contains $k+2r-2$ points while $Y(-r)$ contains at least $\nu$ points, it follows from (\ref{PW2.sec.4.lem.1.eq.2}) that
\begin{equation}\label{PW2.sec.4.lem.1.eq.4}
\bigl|Y(-r)\mathord{\setminus}Y(k+r)\bigr|\geq\frac{7\nu}{9}-k+2\mbox{.}
\end{equation}

By the same argument,
\begin{equation}\label{PW2.sec.4.lem.1.eq.5}
\bigl|Y(k+r)\mathord{\setminus}Y(-r)\bigr|\geq\frac{7\nu}{9}-k+2\mbox{.}
\end{equation}

Finally, let $s$ be the largest integer such that the intersection of $Y(k+r+s)$ and $Y(-r-s)$ is contained in an open half-plane bounded by a line through the origin of $\mathbb{R}^2$. We will assume without loss of generality that this half-plane is the upper half-plane by using if needed an affine transformation of $\mathbb{R}^2$ that fixes $V$ pointwise. Denote by $L^-$ the line through $o$ and $\rho^{-r}(o)$. Likewise, let $L^+$ be the line through $o$ and $\rho^{k+r}(o)$. The resulting situation is depicted on the right of Fig.~\ref{PW2.sec.4.fig.4}. The lines $H$, $V$, $L^-$, and $L^+$ cut $\mathbb{R}^2$ in eight open cones. We shall use the same strategy as in the proof of Lemma \ref{PW2.sec.4.lem.1.5}: it suffices to show that there is a constant fraction of the points in $X$ that lie in $Y(-r)\mathord{\setminus}Y(1)$, in $Y(k+r)\mathord{\setminus}Y(k+1)$, and in the intersections of $Y(-r)$ and $Y(k+r)$ with the open lower half-plane. When $X$ is large enough, this will allow to pick the right amount of points in each of these sets to serve as vertices of the desired subpolygon of $\mathrm{P}$, and to complete this set of vertices with $o$ and $c$. According to (\ref{PW2.sec.4.lem.1.eq.0}) and (\ref{PW2.sec.4.lem.1.eq.2}), each of the former two sets contains a constant fraction of the points in $X$ and we shall focus on the latter two.

By construction, the upper open half-plane now contains exactly $k+2(r+s)$ points of $X$. As in addition, $\mathrm{P}$ is balanced with respect to $o$,
$$
s\leq\frac{3\nu}{4}-\frac{k}{2}-r
$$
and combining this with (\ref{PW2.sec.4.lem.1.eq.2}) yields
\begin{equation}\label{PW2.sec.4.lem.1.eq.6}
s\leq\frac{23\nu}{36}-\frac{k}{2}+1\mbox{.}
\end{equation}

However, the upper half-plane contains exactly $s+1$ points of $Y(-r)\mathord{\setminus}Y(k+r)$ and $s+1$ points of $Y(k+r)\mathord{\setminus}Y(-r)$. Hence by (\ref{PW2.sec.4.lem.1.eq.4}), (\ref{PW2.sec.4.lem.1.eq.5}), and  (\ref{PW2.sec.4.lem.1.eq.6}), the number of points of $Y(-r)$ and $Y(k+r)$ in the lower open half-plane is at least
$$
\frac{5\nu}{36}-\frac{k}{2}
$$

This can be lower bounded according to (\ref{PW2.sec.4.lem.1.eq.1}) as
$$
\frac{5\nu}{36}-\frac{k}{2}\geq\frac{\nu}{72}-1\mbox{.}
$$

Therefore, by (\ref{PW2.sec.4.lem.1.eq.0}), at least a constant fraction of the points from $X$ in the lower half-plane belong to $Y(-r)$ and to $Y(k+r)$, as desired.
\end{proof}

Using Lemmas \ref{PW2.sec.4.lem.1.5} and \ref{PW2.sec.4.lem.1} jointly with Theorem~\ref{PW2.sec.4.lem.2}, we can now establish the three statements that we have announced in the introduction.


\begin{proof}[Proof of Theorems \ref{PW2.sec.1.thm.0}, \ref{PW2.sec.1.thm.1}, and \ref{PW2.sec.1.thm.2}]
First consider a simple polygon $\mathrm{P}$ with vertex set $X$. Under the assumptions of Theorem \ref{PW2.sec.1.thm.2}, it follows from Lemma~\ref{PW2.sec.4.lem.1} that $\mathrm{P}$ admits, up to an affine transformation of $\mathbb{R}^2$, a polygon $\mathrm{Q}$ in $\mathcal{Q}$ as a subpolygon. Similarly, if $\mathrm{P}$ is a convex polygon and $p$ a puncture of $\mathrm{P}$ that satisfy the requirements of Theorem \ref{PW2.sec.1.thm.1} then applying Lemma~\ref{PW2.sec.4.lem.1} after cutting away from $\mathrm{P}$ a triangle incident to $p$ and to an edge of $\mathrm{P}$ provides a polygon $\mathrm{Q}$ from $\mathcal{Q}$ whose vertices form a subset of $X\cup\{p\}$ up to an affine transformation of $\mathbb{R}^2$. Lastly, if $\mathrm{P}$ is a convex polygon then under the requirement from Theorem \ref{PW2.sec.1.thm.0} that $\mathrm{P}$ has sufficiently many vertices, it follows from Lemma \ref{PW2.sec.4.lem.1.5} that one can place a puncture $p$ in $\mathrm{P}$ in such a way that some polygon $\mathrm{Q}$ from $\mathcal{Q}$ has for its vertices a subset of $X\cup\{p\}$ up to an affine transformation of $\mathbb{R}^2$.

In all three cases, cutting $\mathrm{Q}$ away from $\mathrm{P}$ results in a (possibly empty) collection of simple polygons. Consider a (possibly empty) set $A$ of arcs that collectively triangulate all the polygons in this collection and denote by $\mathcal{G}$ the subgraph induced in $\mathcal{F}(\mathrm{P},Z)$ by the triangulations that admit $A$ as a subset, where $Z$ is equal to $X$ when $\mathrm{P}$ is non-convex and to $X\cup\{p\}$ when $\mathrm{P}$ is convex. By construction, there is a natural isomorphism
$$
\varphi:\mathcal{F}(\mathrm{Q},Z\cap\mathrm{Q})\rightarrow\mathcal{G}
$$
that sends $T$ to $A\cup{T}$. According to Theorem~\ref{PW2.sec.4.lem.2}, there exists an arc $\varepsilon$ and a geodesic path in $\mathcal{F}(\mathrm{Q},Z\cap\mathrm{Q})$ that starts and ends in the subgraph $\mathcal{F}_\varepsilon(\mathrm{Q},Z\cap\mathrm{Q})$ but is not entirely contained in that subgraph. If $\mathcal{G}$ is strongly convex in $\mathcal{F}(\mathrm{P},Z)$ then $\varphi$ sends that path to a geodesic path in $\mathcal{F}(\mathrm{P},Z)$ that starts and ends in $\mathcal{F}_\varepsilon(\mathrm{P},Z)$ but is not entirely contained in that subgraph, which proves the theorem. If however, $\mathcal{G}$ is not strongly convex in $\mathcal{F}(\mathrm{P},Z)$, then there exists an arc $\varepsilon$ in $A$ such that $\mathcal{F}_\varepsilon(\mathrm{P},Z)$ is not strongly convex in $\mathcal{F}(\mathrm{P},Z)$.
\end{proof}

\begin{figure}[b]
\begin{centering}
\includegraphics[scale=1]{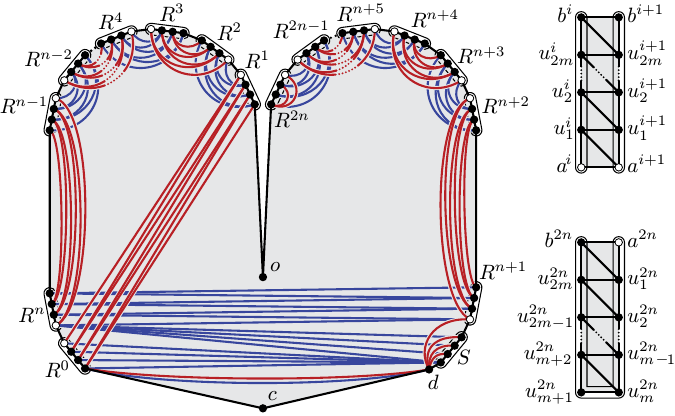}
\caption{The triangulations $T^-$ and $T^+$.}\label{PW2.sec.4.fig.3}
\end{centering}
\end{figure}

We now fix a polygon $\mathrm{Q}$ from $\mathcal{Q}$ and describe the two triangulations $T^-$ and $T^+$ of $\mathrm{Q}$ that will allow to prove Theorem~\ref{PW2.sec.4.lem.2}. These two triangulations are sketched in Fig. \ref{PW2.sec.4.fig.3}, glued along their boundary and thus forming a triangulated sphere. In the figure, $T^-$ is shown below $T^+$ and not all the interior arcs of these triangulations are represented for the sake of clarity. Let us describe $T^-$ and $T^+$ in more details. In $T^-$, all the subpolygons with vertex set $R^i\cup{R^{i+1}}$ where $i$ is odd are triangulated using the zigzag triangulation shown top right on Fig.~\ref{PW2.sec.4.fig.3}. The interior arcs of a zigzag triangulation form a simple path that alternates between left and right turns. Note that the two extremities of that path are $b^i$ and $a^{i+1}$ while $a^i$ and $b^{i+1}$ are not incident to it. Similary, the subpolygons with vertex set $R^i\cup{R^{i+1}}$ where $i$ is even are triangulated in $T^+$ using the same zigzag triangulation. The subpolygon with vertex set $R^{2n}$ will be further triangulated in $T^+$ with the zigzag triangulation shown bottom right on Fig. \ref{PW2.sec.4.fig.3} whose alternating simple path starts at $b^{2n}$, ends at $u^{2n}_m$ and does not touch $a^{2n}$ or $u^{2n}_{m+1}$. A key point of this construction is that the triangulations of the subpolygons with vertex set $R^i\cup{R^{i+1}}$ contained in $T^-$ and in $T^+$, together with the triangulation of the subpolygon with vertex set $R^{2n}$ in $T^+$ form a long ribbon that alternates between the bottom and the top of the triangulated sphere. This ribbon is unfolded on the left of Fig. \ref{PW2.sec.6.fig.1} and one can see that most of the vertices that it contains are incident to six arcs.

Let us now complete our description of $T ^-$. First, we connect the vertex $a^n$ (the last vertex of $R^n$ counter-clockwise around $\mathrm{Q}$) to every vertex in $S$ by an arc of $T^-$. Likewise, the last vertex of $S$ clockwise around $\mathrm{Q}$ (which does not have a label) is connected to every vertex in $R^0$ by an arc of $T^-$. Moreover there is an arc of $T^-$ between $b^0$ (the last vertex of $R^0$ counter-clockwise around $\mathrm{Q}$) and $d$. In order to complete this set of arcs into a triangulation of $\mathrm{Q}$, there remains to triangulate the subpolygon of $\mathrm{Q}$ with vertices $o$ and $b^1$ to $b^{2n}$ (which is untriangulated in Fig. \ref{PW2.sec.4.fig.3}): we do this by adding to $T^-$ an arc between $o$ and each $b^i$. These arcs are not shown in the figure for readability.

In order to complete our description of $T^+$, we add to the zigzag triangulations an arc between $d$ and every vertex in $S$ except for the one that is consecutive to $d$ which is already connected to $d$ by an edge of $\mathrm{Q}$. We also add an arc between $d$ and  $a^{n+1}$ (the last vertex of $R^{n+1}$ clockwise around $\mathrm{Q}$), and the same arc between $b^0$ and $d$ as in $T^-$ (the two arcs are not shown on top of one another in the figure as we think of them as arcs with the same pair of vertices embedded in a sphere). There remains to triangulate the subpolygon with vertices $a^0$ to $a^n$ and the subpolygon with vertices $o$, $d$, $b^0$, $b^1$, $b^{2n}$, and $a^{n+1}$ to $a^{2n}$. Note that both of these subpolygons are untriangulated in Fig.~\ref{PW2.sec.4.fig.3} for readability. We triangulate the latter subpolygon within $T^+$ by adding arcs incident to $o$ and whose other extremities are $b^0$, $d$, and $a^{n+1}$ to $a^{2n}$. For the former subpolygon, we use any triangulation of it within $T^+$ as our result will not depend on this choice. We can bound the distance of $T^-$ and $T^+$ in $\mathcal{F}(\mathrm{Q},X)$ as follows, where $X$ stands, from now on, for the vertex set of $\mathrm{Q}$.

\begin{lem}\label{PW2.sec.4.lem.3}
The distance of $T^-$ and $T^+$ in $\mathcal{F}(\mathrm{Q},X)$ is at most
$$
|S|+8n(m+1)+4m+6\mbox{.}
$$
\end{lem}
\begin{proof}
It suffices to describe paths in $\mathcal{F}(\mathrm{Q},X)$ from $T^-$ and $T^+$ that end in a same triangulation $T^\star$ and whose lengths sum to the desired bound.

Starting with $T^-$, we flip the arc incident to $d$ and $b^0$, which introduces the arc $\varepsilon$ incident to $c$ and to the vertex $x$ of $S$ that comes last clockwise around $\mathrm{Q}$. Then, we flip all the remaining interior arcs of $T^-$ that are incident to a vertex in $R^0$, in the clockwise order of these vertices around $\mathrm{Q}$. Note that each of these flips introduces an arc incident to $c$. Further flipping $\varepsilon$ in the resulting triangulation introduces the arc $\eta$ with vertices $d$ and $a^n$. We sketch the triangulation $T$ obtained from that sequence of flips in Fig. \ref{PW2.sec.4.fig.2} together with (on top) the interior arcs of $T^+$ incident to $d$ and whose other extremity is $a^{n+1}$ or a vertex of $S$. One can see that all of these arcs of $T^+$ can be introduced in $T$ using a sequence of $|S|$ flips: it suffices to flip the arcs of $T$ incident to a vertex in $S$, in the counter-clockwise order of these vertices around $\mathrm{Q}$. We further flip $\eta$, which introduces the arc with vertices $c$ and $a^{n+1}$.

\begin{figure}
\begin{centering}
\includegraphics[scale=1]{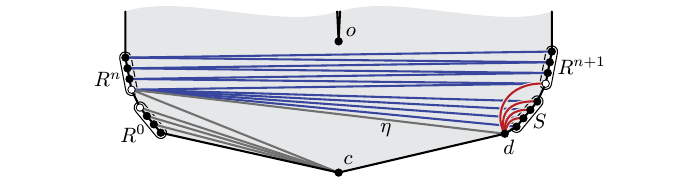}
\caption{The intermediate triangulation $T$.}\label{PW2.sec.4.fig.2}
\end{centering}
\end{figure}

Since $\mathrm{Q}$ is star-shaped with respect to $c$, we can flip all of the remaining $4n(m+1)-1$ arcs of $T^-$ in such a way that each flip introduces an arc incident to $c$, resulting in a triangulation $T^\star$ of $\mathrm{Q}$ whose interior arcs are all incident to either $c$ or $d$. The interior arcs of $T^\star$ incident to $d$ are precisely the ones in $T^+$ between $d$ and either $a^{n+1}$ or a vertex in $S$. Counting the flips that we have performed shows that the distance of $T^-$ and $T^\star$ in $\mathcal{F}(\mathrm{Q},X)$ is at most
\begin{equation}\label{PW2.sec.4.lem.3.eq.1}
|S|+4n(m+1)+2m+4\mbox{.}
\end{equation}

Now observe that $T^+$ can also be changed into $T^\star$ by flipping the arc incident to $d$ and $b^0$ and then all the other interior arcs of $T^+$ that are not incident to $d$ in such a way that each flip introduces an arc incident to $c$. As the number of interior arcs in $T^+$ that are not incident to $d$ is $2(2n+1)(m+1)-1$, this shows that the distance between $T^+$ and $T^\star$ in $\mathcal{F}(\mathrm{Q},X)$ is at most
$$
2(2n+1)(m+1)\mbox{.}
$$

Summing this with (\ref{PW2.sec.4.lem.3.eq.1}) completes the proof.
\end{proof}

\section{An application of the decomposition lemma}\label{PW2.sec.5}

In this section and the next, we consider a polygon $\mathrm{Q}$ from $\mathcal{Q}$ and denote by $X$ its vertex set. We further consider the two triangulations $T^-$ and $T^+$ of $\mathrm{Q}$ that we have built in Section \ref{PW2.sec.4} and a blow-up triangulation $K$ of $\mathrm{Q}$ that corresponds to a geodesic path from $T^-$ to $T^+$ in $\mathcal{F}(\mathrm{Q},X)$.

We aim at showing that, if $n$ and $m/n$ are both large enough, then $K$ contains two different arcs with the same pair of vertices, which will prove Theorem \ref{PW2.sec.4.lem.2}. From now on, we assume for contradiction that there is at most one arc of $K$ incident to any two vertices of $\mathrm{Q}$. As $T^-$ and $T^+$ both contain the arc with vertices $b^0$ and $d$, this implies that $K$ is pinched along a triangle with vertices $c$, $b^0$, and $d$. In other words, the union of the faces of $K$ is a topological ball to which a triangle incident to $c$ has been glued along an arc with vertices $b^0$ and $d$. Hence that triangle can be cut away from $K$ without changing the number of tetrahedra in $K$. The first step of our proof of Theorem \ref{PW2.sec.4.lem.2} is to further cut $K$ into smaller blow-up triangulations. This will be done via our decomposition lemma for a topological circle whose image by $\pi$ is the triangle with vertices $a^n$, $a^{n+1}$, and $d$. In order to do that, let us first remark that, along the path between $T^-$ and $T^+$ built in the proof of Lemma~\ref{PW2.sec.4.lem.3}, one of the triangulations contains the arc with extremities $a^n$ and $d$. In fact, by the following lemma, this is the case of every geodesic path between $T^-$ and $T^+$.

\begin{lem}\label{PW2.sec.5.lem.1}
Some flip exchanges an arc of $T^-$ incident to $a^n$ and an arc of $T^+$ incident to $d$ along any geodesic path between $T^-$ and $T^+$ in $\mathcal{F}(\mathrm{Q},X)$.
\end{lem}
\begin{proof}
Consider a path from $T^-$ to $T^+$ in $\mathcal{F}(\mathrm{Q},X)$ and assume that no flip along that path exchanges one of the arcs of $T^-$ incident to $a^n$ with an arc of $T^+$ incident to $d$. In that case, the flips that remove each of the interior arcs of $T^-$ incident to $a^n$ and the flips that introduce each of the interior arcs of $T^+$ incident to $d$ (other than the one whose other extremity is $b^0$ which already belongs to $T^-$) are all distinct. As there are $2|S|+1$ such arcs in total that need to be removed from $T^-$ or introduced in $T^+$, this implies that the considered path has length at least $2|S|+1$. However, by Definition \ref{PW2.sec.4.defn.1},
$$
|S|\geq8n(m+1)+4m+6\mbox{.}
$$

Hence, the length of the considered path is at least
$$
|S|+8n(m+1)+4m+7
$$
and by Lemma \ref{PW2.sec.4.lem.3}, this path cannot be a geodesic in $\mathcal{F}(\mathrm{Q},X)$.
\end{proof}

Consider a quadrilateral whose diagonals are an arc of $T^-$ incident to $a^n$ and an arc of $T^+$ incident to $d$. Observe that any such quadrilateral admits an edge $\varepsilon$ with extremities $a^n$ and $d$. Hence if a flip exchanges an arc of $T^-$ incident to $a^n$ with an arc of $T^+$ incident to $d$ in a triangulation of $\mathrm{Q}$, then that triangulation must contain $\varepsilon$. In particular, as $K$ corresponds to a geodesic path between $T^-$ and $T^+$ in $\mathcal{F}(\mathrm{Q},X)$, it follows from Lemma \ref{PW2.sec.5.lem.1} that it contains an arc $\beta$ with extremities $a^n$ and $d$. However, $K$ also contains an arc $\alpha$ with vertices $a^n$ and $a^{n+1}$ that is not above any face of $K$ and an arc $\gamma$ with vertices $d$ and $a^{n+1}$ that is not below any face of $K$. This is because $T^-$ and $T^+$ contain arcs with these pairs of vertices. The union of $\alpha$, $\beta$, and $\gamma$ is a topological circle and we shall apply our decomposition lemma to that circle.

\begin{lem}\label{PW2.sec.5.lem.2}
$K$ contains a triangle whose boundary is $\alpha\cup\beta\cup\gamma$.
\end{lem} 
\begin{proof}
Our assumption that $K$ does not contain two different arcs with the same pair of vertices implies that a triangle in $K$ whose image by $\pi$ admits $\pi(\alpha\cup\beta\cup\gamma)$ as its boundary must be incident to $\alpha$, $\beta$, and $\gamma$. Therefore, by Lemma \ref{PW2.sec.3.lem.5}, it suffices to show that every lower and upper tetrahedron of $K$ that penetrates $\alpha\cup\beta\cup\gamma$ can be pulled to $a^n$ and $d$, respectively. Recall that, by Proposition~\ref{PW2.sec.3.prop.3}, these tetrahedra only depend on $\alpha$, $\beta$, and $\gamma$. 

First consider a lower tetrahedron $\sigma$ of $K$ that penetrates $\alpha\cup\beta\cup\gamma$ and denote by $\tau$ its base. According to Proposition~\ref{PW2.sec.3.prop.3}, $\tau$ does not have a vertex in $\mathrm{Q}^\alpha\mathord{\setminus}\{a^{n+1}\}$. As a consequence, the convex hull of $\pi(\tau)\cup\{a^n\}$ is contained in the subpolygon of $\mathrm{Q}$ with vertex set $R^0\cup{S}\cup\{c,d,a^n,a^{n+1}\}$. As this subpolygon is convex and its intersection with $X$ is precisely its vertex set, $\sigma$ can be pulled to $a^n$ in the sense of Definition \ref{PW2.sec.3.defn.3}, as desired.

Now consider an upper tetrahedron $\sigma$ of $K$ that penetrates $\alpha\cup\beta\cup\gamma$. Denote by $\tau$ the base of $\sigma$ and assume for contradiction that the convex hull of $\pi(\tau)\cup\{d\}$ is not contained in $\mathrm{Q}$ or contains a point from $X$ that is not one of its vertices. Note that this can only happen when $\tau$ has a vertex in $R^i$ where $i$ is positive and at most $n-1$. In order for $\sigma$ to penetrate $\alpha\cup\beta\cup\gamma$, that tetrahedron must then be above $\alpha$ and below $\beta$. In particular, the interior of $\sigma$ is non-disjoint from the interior of any triangle incident to $\beta$ and to an arc that admits $a^n$ and a point from $S$ as its vertices but is not above any face of $K$. Hence, $\sigma$ cannot coexist in $K$ with such a triangle. As by Lemma~\ref{PW2.sec.5.lem.1}, $K$ must contain a triangle exactly like that one, the convex hull of $\pi(\tau)\cup\{d\}$ is contained in $\mathrm{Q}$ and its intersection with $X$ is equal to its vertex set, as desired.
\end{proof}

\begin{figure}[b]
\begin{centering}
\includegraphics[scale=1]{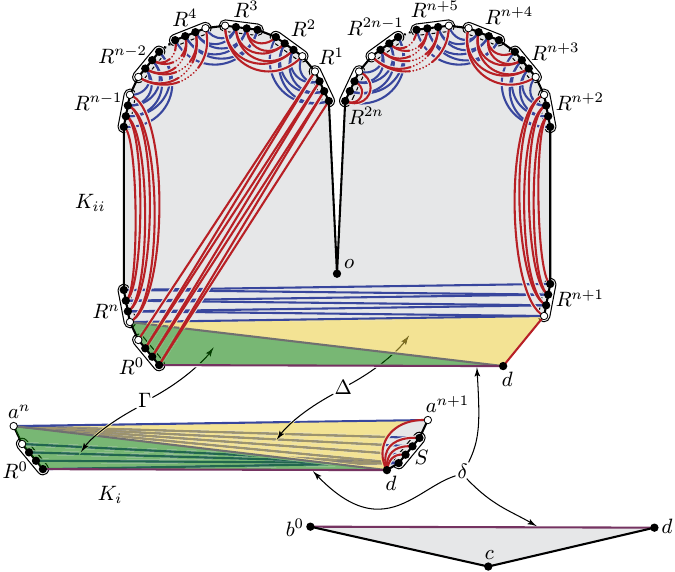}
\caption{A decomposition of $K$.}\label{PW2.sec.5.fig.1}
\end{centering}
\end{figure}

Lemma \ref{PW2.sec.5.lem.2} makes it possible to apply the decomposition strategy already used in \cite{PourninWang2021}. Denote by $\Delta$ the triangle of $K$ with boundary $\alpha\cup\beta\cup\gamma$ provided by this lemma. Further denote by $\delta$ the arc of $K$ with vertices $b^0$ and $d$ and recall that the union of the faces of $K$ is a $3$-dimensional topological ball with a triangle glued along $\delta$. Consider a (necessarily geodesic) path $(T_i)_{0\leq{i}\leq{k}}$ in $\mathcal{F}(\mathrm{Q},X)$ from $T^-$ to $T^+$ that corresponds to $K$ and the family $(\Pi_i)_{0\leq{i}\leq{k}}$ of topological disks that Proposition~\ref{PW2.sec.2.prop.1} associates with $K$ and $(T_i)_{0\leq{i}\leq{k}}$. Let $q$ be any index such that $\beta$ is contained in $\Pi_q$ and denote by $\Gamma$ the set of the triangles of $K$ contained in $\Pi_q$ that lie between $\beta$ and $\delta$.

Now let us cut $K$ along $\delta$, $\Delta$, and all the triangles in $\Gamma$. This results in the triangle and the two blow-up triangulations $K_i$ and $K_{ii}$ sketched in Fig. \ref{PW2.sec.5.fig.1}. Denoting by $\mathrm{Q}_i$ the subpolygon of  $\mathrm{Q}$ with vertex set
$$
\{a^n,a^{n+1},d\}\cup{R^0}\cup{S}\mbox{,}
$$
$K_i$ is a blow-up triangulation of $\mathrm{Q}_i$ while $K_{ii}$ is a blow-up triangulation of the subpolygon $\mathrm{Q}_{ii}$ of $\mathrm{Q}$ whose vertex set is obtained from that of $\mathrm{Q}$ by removing $c$ and all the vertices contained in $S$. Fig. \ref{PW2.sec.5.fig.1} does not show which triangles are in $\Gamma$ as we do not know which they are. One can see that the $|S|+2(m+1)$ arcs of $K_i$ that are not above any face of $K_i$ and whose image by $\pi$ is not an edge of $\mathrm{Q}_i$ are below one of the two arcs of $K_i$ incident to $d$ and whose other vertex is $a^n$ or $a^{n+1}$. As a consequence, $K_i$ contains at least $|S|+2(m+1)$ tetrahedra. Indeed a tetrahedron in $K_i$ cannot be above and incident to two different arcs of $K_i$. However, according to Lemma \ref{PW2.sec.4.lem.3}, $K$ contains at most
$$ 
|S|+8n(m+1)+4m+6
$$
tetrahedra. Moreover, by construction, the set of the tetrahedra of $K$ is the disjoint union of the sets of the tetrahedra of $K_i$ and $K_{ii}$. We immediately obtain the following bound as a consequence of these observations.

\begin{prop}\label{PW2.sec.5.prop.1}
$K_{ii}$ contains at most $8n(m+1)+2m+4$ tetrahedra.
\end{prop}

\section{A hyperbolic volume argument}\label{PW2.sec.6}

In this section, we embed the boundary of the blow-up triangulation $K_{ii}$ constructed in Section \ref{PW2.sec.5} in the $3$-dimensional hyperbolic space $\mathbb{H}^3$ as the boundary of an ideal polytope $\Xi$ of large volume. In particular, the number of ideal tetrahedra required to triangulate $\Xi$ will be larger than the estimate provided by Proposition \ref{PW2.sec.5.prop.1}, contradicting our assumption that two different arcs of $K$ never have the same pair of vertices and thus proving Theorem \ref{PW2.sec.4.lem.2}.

\begin{figure}
\begin{centering}
\includegraphics[scale=1]{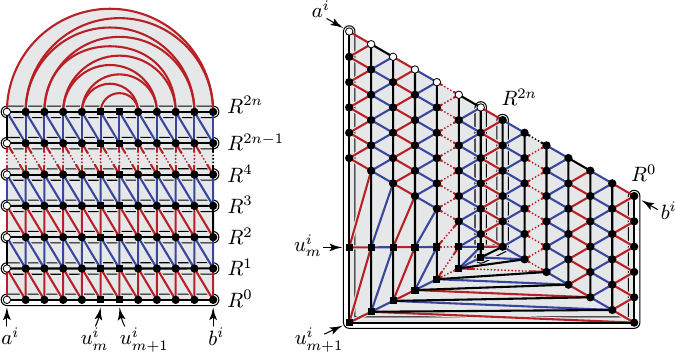}
\caption{The ribbon in the boundary of $K$ (left) and its embedding in $\mathbb{H}^3$ (right). Only $R^0$ and $R^{2n}$ are marked on the right in order not to overburden the figure.}\label{PW2.sec.6.fig.1}
\end{centering}
\end{figure}

We describe our embedding of the boundary of $K_{ii}$ in $\mathbb{H}^3$ using the Poincar{\'e} upper half-space model of $\mathbb{H}^3$, denoting by $P_\infty$ and $p_\infty$ the plane and the point at infinity, respectively. Recall that the boundary of $K$ contains the ribbon that is unfolded on the left of Fig.~\ref{PW2.sec.6.fig.1}. By construction, this ribbon is also contained in the boundary of $K_{ii}$ and we begin by embedding it in $\mathbb{H}^3$. The vertices of that ribbon are placed in $P_\infty$ as shown on the right of Fig.~\ref{PW2.sec.6.fig.1}. The arcs and triangles of the ribbon incident to these vertices are embedded as half-circles and spherical triangles in the upper half-space and one can think of this representation as a view from below $P_\infty$ when looking in direction of $p_\infty$. Note that, except for the triangles incident to $u^i_m$ or $u^i_{m+1}$ (which are marked with squares in the figure), all the triangles in the embedded ribbon are equilateral, which will be instrumental when estimating the volume of $\Xi$. The points $u_m^0$ to $u_m^{2n}$ and $u_{m+1}^0$ to $u_{m+1}^{2n}$ are placed along two straight lines but the precise placement of these points is not important as long as the orthogonal projections on $P_\infty$ of distinct triangles do not have overlapping interiors. In particular we will assume in the following, by adjusting the slope of these two lines, that the distance from $u^0_m$ to $u^0_{m-1}$ and $u^0_{m+1}$ is fixed independently from $n$.

\begin{figure}
\begin{centering}
\includegraphics[scale=1]{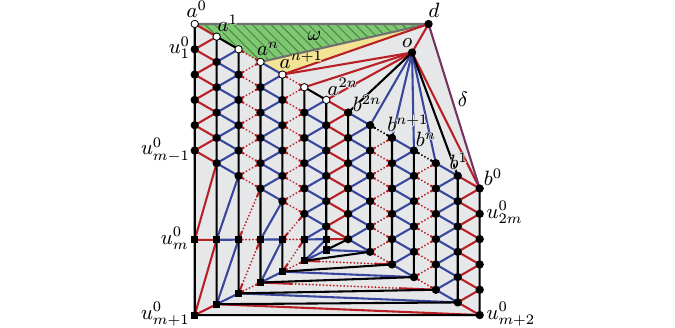}
\caption{The lower faces of $\Xi$.}\label{PW2.sec.6.fig.2}
\end{centering}
\end{figure}

In order to complete embedding the boundary of $K_{ii}$ within $\mathbb{H}^3$, there remains to place the vertices $o$ and $d$. These points will be placed in $P_\infty$ as sketched in Fig.~\ref{PW2.sec.6.fig.2}. Here, again, the precise placement of $o$ and $d$ within $P_\infty$ is not too important, except for the requirement that the triangles shown in the figure do not overlap in their interior. In particular, we will assume that both $o$ and $d$ are placed in the plane $P_\infty$ close to the interior of the line segment with vertices $a^{2n}$ and $b^{2n}$ in such a way that the line through $d$ and $a^n$ in  that plane intersects the interior of the line segment with extremities $a^0$ and $u_1^0$. The announced hyperbolic ideal polytope $\Xi$ has for its vertices our embedding of
$$
\bigcup_{i=0}^{2n}R^i\cup\{d,o\}
$$
in $P_\infty$. Its edges and triangular faces are the half-circles and spherical triangles in the upper half-space with the same vertices as the arcs and triangles in the boundary of $K_{ii}$. Let us describe the facial structure of that polytope in order to verify that its boundary is indeed a topological sphere.

Fig.~\ref{PW2.sec.6.fig.2} shows the triangular faces of $\Xi$ that can be seen from below $P_\infty$, which we refer to as the \emph{lower} faces of $\Xi$. It should be noted that $\Xi$ may have a lower face with vertices $d$, $a^0$, and $a^n$ (striped in the figure) depending on whether the corresponding triangle lies or not in the boundary of $K_{ii}$. Indeed recall that $K_{ii}$ was obtained by cutting $K$ along $\delta$, $\Delta$, and an unspecified set $\Gamma$ of triangles that forms a triangulation of a topological disk with vertex set
$$
R^0\cup\{a^n,d\}\mbox{.}
$$

If there is a triangle $\omega$ with vertices $d$, $a^0$, and $a^n$ in $\Gamma$, that triangle corresponds to the lower face of $\Xi$ that is striped in Fig.~\ref{PW2.sec.6.fig.2}. In that case, the orthogonal projection of $\Xi$ back on $P_\infty$ is a pentagon with vertices $a^0$, $u^0_{m+1}$, $u^0_{m+2}$, $b^0$, and $d$. If however, there is no triangle with vertices $d$, $a^0$, and $a^n$ in $\Gamma$, then, by our assumption that the line in $P_\infty$ through $d$ and $a^n$ intersects the interior of the line segment with extremities $a^0$ and $u_1^0$, the orthogonal projection of $\Xi$ on $P_\infty$ is a simple, non-convex hexagon that admits $a^n$ as its only reflex vertex. In both cases, $\Xi$ has an edge with vertices $a^0$ and $a^n$ which, together with the edges with vertices $a^{i-1}$ and $a^i$ when $i$ ranges from $1$ to $n$, bounds a triangulated disk orthogonal to $P_\infty$ within the upper half-space. The way this disk is triangulated corrresponds to how the polygon with vertices $a^0$ to $a^n$ is triangulated in $T^+$ (which we did not specify). The faces of $\Xi$ that correspond to the triangles in $\Gamma$ (except $\omega$ if that triangle belongs to $\Gamma$) form the rest of the boundary of $\Xi$. Some of these faces may be orthogonal to $P_\infty$ which happens precisely when their three vertices belong to either
$$
\{a^0, u^0_1,\ldots,u^0_{m+1}\}
$$
or
$$
\{u^0_{m+1}\ldots, u^0_{2m},b^0\}\mbox{.}
$$

All the other faces of $\Xi$ will be called its \emph{upper} faces and since the points in $R^0\cup\{a^n,d\}$ have been placed in $P_\infty$ in the same clockwise order as in $\mathrm{Q}$, the projections of the upper faces of $\Xi$ on $P_\infty$ form a triangulation.

It follows from this description that the hyperbolic volume of $\Xi$ is the difference between the volume of the portion of $\mathbb{H}^3$ above its lower faces and the volume of the portion of $\mathbb{H}^3$ above its upper faces. In order to estimate this volume as a function of $n$ and $m$, let us recall that the hyperbolic volume of an ideal tetrahedron $\sigma$ with vertices $p_\infty$, $x$, $y$, and $z$ is
\begin{equation}\label{PW2.sec.6.eq.0}
\mathrm{vol}(\sigma)=\lob(\theta_x)+\lob(\theta_y)+\lob(\theta_z)
\end{equation}
where $\theta_x$, $\theta_y$, and $\theta_z$ are the three internal angles of the triangle with vertices $x$, $y$, and $z$ in $P_\infty$ and $\lob$ stands for Milnor's Lobachevsky function \cite{Milnor1982}:
$$
\lob(\theta)=-\int_0^\theta\log|2\sin u|du\mbox{.}
$$

Moreover $\mathrm{vol}(\sigma)$ is maximal if and only if $x$, $y$, and $z$ are the vertices of an equilateral triangle in the plane at infinity. In other words,
\begin{equation}\label{PW2.sec.6.eq.1}
\mathrm{vol}(\sigma)\leq3\lob\biggl(\frac{\pi}{3}\biggr)
\end{equation}
with equality when $x$, $y$, and $z$ are the vertices of an equilateral triangle.

\begin{lem}\label{PW2.sec.6.lem.1}
Consider a positive number $\epsilon$. If $n$ is large enough, then
$$
\mathrm{vol}(\Xi)\geq3\bigl(2(4n+1)(m-1)-\epsilon{m}-2\bigr)\lob\biggl(\frac{\pi}{3}\biggr)\mbox{.}
$$
\end{lem}
\begin{proof}
By construction, $2(4n+1)(m-1)$ lower faces of $\Xi$ are equilateral. Hence, the volume $v^-$of the portion of $\mathbb{H}^3$ above the lower faces of $\Xi$ satisfies
$$
v^-\geq6(4n+1)(m-1)\lob\biggl(\frac{\pi}{3}\biggr)
$$
and since the hyperbolic volume of $\Xi$ is obtained by subtracting to $v^-$ the volume $v^+$ of the portion of $\mathbb{H}^3$ above its upper faces, it suffices to show that $v^+$ is at most $3(\epsilon m+2)\lob(\pi/3)$ when $n$ is large enough.

Recall that the vertices of all the upper faces of $\Xi$ belong to
$$
R^0\cup\{a^n,d\}
$$
and note that at most one upper face of $\Xi$ has exactly one vertex in
$$
R^0_a=\{a^0, u^0_1,\ldots,u^0_{m+1}\}
$$
and exactly one vertex in
$$
R^0_b=\{u^0_{m+2}\ldots, u^0_{2m},b^0\}\mbox{.}
$$

Indeed the orthogonal projections of any two such triangles on $P_\infty$ have non-disjoint interiors whereas we know that projecting the upper triangles of $\Xi$ orthogonally on $P_\infty$ results a triangulation. Likewise there is at most one upper triangle of $\Xi$ with at most one vertex in $R^0$ as such a triangle must be incident to the edge of $\Xi$ with extremities $a^n$ and $d$. Hence, all the upper triangles of $\Xi$ except at most two have two vertices in $R^0_a$ or two vertices in $R^0_b$.

Consider an upper face $\tau$ of $\Xi$ incident to two points $u_j^0$ and $u_k^0$ from $R^0_a$ and to a point $x$ in $R^0_b\cup\{a^n,d\}$. We assume that $j$ is less than $k$ and for simplicity that $u^0_0$ refers to $a^0$. Observe that the volume of the portion of $\mathbb{H}^3$ above $\tau$ is not greater than the volume of the portion of $\mathbb{H}^3$ above the triangles $\tau_i$ with vertex set $\{u_i^0,u_{i+1}^0,x\}$ where $i$ ranges from $j$ to $k-1$. However, when $n$ goes to infinity, the internal angle at $x$ of the projection of $\tau_i$ on $P_\infty$ goes to $0$ because $x$ belongs to $R^0_b\cup\{a^n,d\}$. Recall in particular that $d$ is placed in $P_\infty$ close to the interior of the line segment with vertices $a^{2n}$ and $b^{2n}$. When $i$ is equal to either $m-1$ or $m$, the property that the internal angle at $x$ goes to $0$ follows from our placement of $u_m^0$ and $u_{m+1}^0$ in such a way that the distance of $u_m^0$ to $u_{m-1}^0$ and $u_m^0$ and $u_{m+1}^0$ is bounded independently from $n$. Since $\lob$ is a continuous function that vanishes at $0$ while $\theta\mapsto\lob(\theta+\pi/2)$ is odd, it follows from (\ref{PW2.sec.6.eq.0}) that if $n$ is large enough, then the volume of the portion of $\mathbb{H}^3$ above any of the triangles $\tau_j$ to $\tau_{k-1}$ is at most $\epsilon\lob(\pi/3)$. Noticing that there are $m+1$ line segments with vertices $u_i^0$ and $u_{i+1}^0$ in $R^0_a$, this provides a bound of
$$
\epsilon(m+1)\lob\biggl(\frac{\pi}{3}\biggr)
$$
on the volume of the portion of $\mathbb{H}^3$ above the upper triangles of $\Xi$ with two vertices in $R^0_a$. Performing a similar estimate for the upper faces of $\Xi$ that are incident to two points from $R^0_b$ shows that, when $n$ is large enough, the hyperbolic volume of the portion of $\mathbb{H}^3$ above them is at most
$$
\epsilon(m-1)\lob\biggl(\frac{\pi}{3}\biggr)\mbox{.}
$$

Adding the volume of the portion of $\mathbb{H}^3$ above the two possible upper triangles of $\Xi$ with at most one vertex in $R^0_a$ and at most one in $R^0_b$ shows that
$$
v^+\leq6\lob\biggl(\frac{\pi}{3}\biggr)+2\epsilon{m}\lob\biggl(\frac{\pi}{3}\biggr)\mbox{.}
$$

As $6+2\epsilon{m}\leq3(2+\epsilon{m})$, this completes the proof.
\end{proof}

Now let us study the set $\mathcal{T}$ of the ideal hyperbolic tetrahedra that correspond to the tetrahedra of $K_{ii}$ via our embedding in $\mathbb{H}^3$. Note that these tetrahedra may have non-disjoint interiors. However, since the tetrahedra in $K_{ii}$ form a triangulation of a topological ball whose boundary is embedded in $\mathbb{H}^3$ as the boundary of $\Xi$, we immediately obtain the following.

\begin{prop}\label{PW2.sec.6.lem.2}
$\displaystyle\sum_{\sigma\in\mathcal{T}}\mathrm{vol}(\sigma)\geq\mathrm{vol}(\Xi)$.
\end{prop}

We now estimate the number of tetrahedra in $\mathcal{T}$ whose volume is at most a fraction of the maximal volume of an ideal hyperbolic tetrahedron.

\begin{lem}\label{PW2.sec.6.lem.2.5}
Consider two numbers $\epsilon$ and $\lambda$ such that $\epsilon$ is positive while $\lambda$ satisfies $0<\lambda<1$. If $n$ is large enough, then the number of tetrahedra contained in $\mathcal{T}$ of volume at most $3\lambda\lob(\pi/3)$ does not exceed
$$
\frac{16n+\epsilon{m}+8}{1-\lambda}\mbox{.}
$$
\end{lem}
\begin{proof}
Denote by $t$ the number of tetrahedra in $\mathcal{T}$ of hyperbolic volume at most $3\lambda\lob(\pi/3)$. According to (\ref{PW2.sec.6.eq.1}), the hyperbolic volume of all the other tetrahedra in $\mathcal{T}$ is at most $3\lob(\pi/3)$ and as a consequence,
$$
\sum_{\sigma\in\mathcal{T}}\mathrm{vol}(\sigma)\leq3\lambda\lob\biggl(\frac{\pi}{3}\biggr)t+3\lob\biggl(\frac{\pi}{3}\biggr)\bigl(|\mathcal{T}|-t\bigr)\mbox{.}
$$

Hence by Lemma \ref{PW2.sec.6.lem.1} and Proposition \ref{PW2.sec.6.lem.2}, 
\begin{equation}\label{PW2.sec.6.lem.2.5.eq.1}
2(4n+1)(m-1)-\epsilon{m}-2\leq(\lambda-1)t+|\mathcal{T}|
\end{equation}
when $n$ is large enough. However, according to Proposition \ref{PW2.sec.5.prop.1}, 
$$
|\mathcal{T}|\leq8n(m+1)+2m+4
$$
and the result follows by combining this with (\ref{PW2.sec.6.lem.2.5.eq.1}).
\end{proof}

Denote by $\mathcal{E}^\star$ the set of the edges of $\Xi$ shared by two equilateral lower faces of $\Xi$. Further denote by $\mathcal{E}^\star_1$ and $\mathcal{E}^\star_2$ the subsets of the edges in $\mathcal{E}^\star$ that are incident to one or two tetrahedra from $\mathcal{T}$, respectively. The number of edges in $\mathcal{E}^\star_1$ is ultimately bounded by a linear function of $n$ and $m$.

\begin{lem}\label{PW2.sec.6.lem.3}
Consider a positive number $\epsilon$. If $n$ is large enough, then
$$
|\mathcal{E}^\star_1|\leq48n+3\epsilon{m}+24\mbox{.}
$$
\end{lem}
\begin{proof}
Observe that the unique tetrahedron $\sigma$ of $\mathcal{T}$ incident to an edge in $\mathcal{E}^\star_1$ has a fixed geometry, which we represent in Fig. \ref{PW2.sec.6.fig.3}. In particular, the hyperbolic volume of $\sigma$ can be computed via (\ref{PW2.sec.6.eq.0}) by subtracting the volume above the upper faces of $\sigma$ from the volume above its lower faces:
$$
\mathrm{vol}(\sigma)=6\lob\biggl(\frac{\pi}{3}\biggr)-2\Biggl(2\lob\biggl(\frac{\pi}{6}\biggr)+\lob\biggl(\frac{2\pi}{3}\biggr)\Biggr)\mbox{.}
$$

\begin{figure}[b]
\begin{centering}
\includegraphics[scale=1]{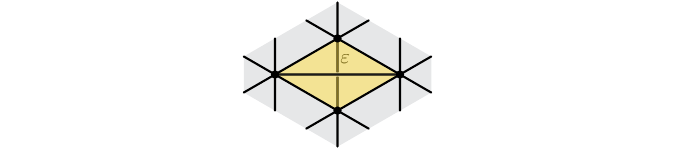}
\caption{The ideal tetrahedron $\sigma$ in Lemma \ref{PW2.sec.6.lem.3}.}\label{PW2.sec.6.fig.3}
\end{centering}
\end{figure}

However, according to \cite[Lemma 1]{Milnor1982},
\begin{equation}\label{PW2.sec.6.lem.3.eq.1}
\lob(2\theta)=2\Biggl(\lob(\theta)+\lob\biggl(\theta+\frac{\pi}{2}\biggr)\Biggr)
\end{equation}
and using this formula with $\theta$ equal to $\pi/6$ yields
$$
\mathrm{vol}(\sigma)=4\lob\biggl(\frac{\pi}{3}\biggr)+2\lob\biggl(\frac{2\pi}{3}\biggr)\mbox{.}
$$

As in addition $\lob$ is odd with period $\pi$,
$$
\mathrm{vol}(\sigma)=2\lob\biggl(\frac{\pi}{3}\biggr)\mbox{.}
$$

Since a tetrahedron in $\mathcal{T}$ cannot be incident to two different edges in $\mathcal{E}^\star_1$, applying Lemma \ref{PW2.sec.6.lem.2.5} with $\lambda$ equal to $2/3$ proves the lemma.
\end{proof}

Let $\mathcal{T}^\star$ be the set of the tetrahedra in $\mathcal{T}$ that are incident to at least one of the equilateral lower face of $\Xi$ and recall that, by Proposition \ref{PW2.sec.5.prop.1}. Thanks to Lemma \ref{PW2.sec.6.lem.3}, we can show that the number of tetrahedra in $\mathcal{T}$ that do not belong to $\mathcal{T}^\star$ is ultimately at most a linear function of $n$ and $m$.

\begin{lem}\label{PW2.sec.6.lem.3.5}
Consider a positive number $\epsilon$. If $n$ is large enough, then
$$
|\mathcal{T}\mathord{\setminus}\mathcal{T}^\star|\leq64n+3\epsilon{m}+30\mbox{.}
$$
\end{lem}
\begin{proof}
As already remarked in the proof of Lemma \ref{PW2.sec.6.lem.1}, $\Xi$ has
\begin{equation}\label{PW2.sec.6.lem.3.5.eq.1}
2(4n+1)(m-1)
\end{equation}
equilateral lower faces. However, the number of tetrahedra in $\mathcal{T}$ that are incident to two such faces is precisely $\mathcal{E}^\star_1$. As no vertex of $\Xi$ is incident to just three faces of $\Xi$, no tetrahedron in $\mathcal{T}$ is incident to more than two faces of $\Xi$.

Hence, when $n$ is large enough, $|\mathcal{T}^\star|$ is at least the difference between~(\ref{PW2.sec.6.lem.3.5.eq.1}) and the bound on $|\mathcal{E}^\star_1|$ provided by Lemma \ref{PW2.sec.6.lem.3}:
\begin{equation}\label{PW2.sec.6.lem.3.5.eq.2}
|\mathcal{T}^\star|\geq2(4n+1)(m-1)-48n-3\epsilon{m}-24\mbox{.}
\end{equation}

However, by Proposition \ref{PW2.sec.5.prop.1},
\begin{equation}\label{PW2.sec.6.lem.3.5.eq.3}
|\mathcal{T}|\leq8n(m+1)+2m+4
\end{equation}
and combining (\ref{PW2.sec.6.lem.3.5.eq.2}) with (\ref{PW2.sec.6.lem.3.5.eq.3}) completes the proof.
\end{proof}

We now show that the number of edges in $\mathcal{E}^\star$ incident to at least three tetrahedra from $\mathcal{T}$ is ultimately bounded by a linear function of $n$ and $m$.

\begin{lem}\label{PW2.sec.6.lem.4}
Consider a positive number $\epsilon$. If $n$ is large enough, then
$$
\bigl|\mathcal{E}^\star\mathord{\setminus}(\mathcal{E}^\star_1\cup\mathcal{E}^\star_2)\bigr|\leq224n+12\epsilon{m}+108\mbox{.}
$$
\end{lem}
\begin{proof}
First note that a tetrahedron contained in $\mathcal{T}\mathord{\setminus}\mathcal{T}^\star$ is incident to at most two edges from $\mathcal{E}^\star$. Hence, by Lemma \ref{PW2.sec.6.lem.3.5}, at most $128n+6\epsilon{m}+60$ edges in $\mathcal{E}^\star$ are incident to a tetrahedron from $\mathcal{T}\mathord{\setminus}\mathcal{T}^\star$ when $n$ is large enough.

Now consider an edge $\varepsilon$ in $\mathcal{E}^\star$ incident to at least three tetrahedra contained in $\mathcal{T}^\star$. Then at least one tetrahedron $\sigma$ in $\mathcal{T}^\star$ is incident to $\varepsilon$ but not to either of the two lower faces of $\Xi$ that share $\varepsilon$ as an edge. Since $\sigma$ belongs to $\mathcal{T}^\star$, this tetrahedron must be incident to an equilateral lower face of $\Xi$. That face necessarily shares a single vertex with $\varepsilon$ and, therefore, the geometry of $\sigma$ must be as shown in Fig.~\ref{PW2.sec.6.fig.4}. We can compute the hyperbolic volume of $\sigma$ using the same strategy as in the proof of Lemma \ref{PW2.sec.6.lem.3} by subtracting the volume above its (unique) upper face from the volume above its lower faces:
$$
\mathrm{vol}(\sigma)=3\lob\biggl(\frac{\pi}{3}\biggr)+2\lob\biggl(\frac{\pi}{6}\biggr)+\lob\biggl(\frac{2\pi}{3}\biggr)-\lob\biggl(\frac{\pi}{6}\biggr)-\lob\biggl(\frac{\pi}{3}\biggr)-\lob\biggl(\frac{\pi}{2}\biggr)\mbox{.}
$$

\begin{figure}[b]
\begin{centering}
\includegraphics[scale=1]{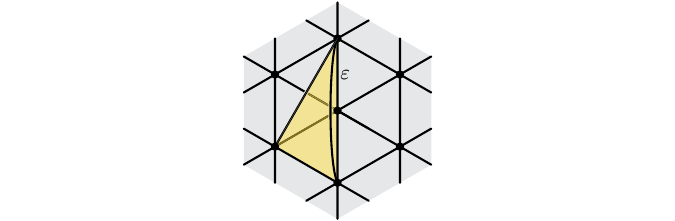}
\caption{The ideal tetrahedron $\sigma$ in Lemma \ref{PW2.sec.6.lem.4}. One of the edges of the upper face of $\sigma$ is bent for clarity.}\label{PW2.sec.6.fig.4}
\end{centering}
\end{figure}

Observing that $\lob$ vanishes at $\pi/2$, and using (\ref{PW2.sec.6.lem.3.eq.1}) with $\theta$ equal to $\pi/6$ yields
\begin{equation}\label{PW2.sec.6.lem.4.eq.1}
\mathrm{vol}(\sigma)=\frac{5}{2}\lob\biggl(\frac{\pi}{3}\biggr)\mbox{.}
\end{equation}

We have shown that for large enough $n$, except for at most $128n+6\epsilon{m}+60$ of them, each edge $\varepsilon$ in $\mathcal{E}^\star\mathord{\setminus}(\mathcal{E}^\star_1\cup\mathcal{E}^\star_2)$ is incident to at least one tetrahedron $\sigma$ in $\mathcal{T}^\star$ that is not incident to either of the two faces of $\Xi$ that share $\varepsilon$. One can see in Fig.~\ref{PW2.sec.6.fig.4} that $\sigma$ cannot play this role for any edge other than $\varepsilon$. Moreover, the hyperbolic volume of $\sigma$ satisfies (\ref{PW2.sec.6.lem.4.eq.1}) and, applying Lemma~\ref{PW2.sec.6.lem.2.5} with $\lambda$ equal to $5/6$ shows that $\mathcal{T}$ contains at most $96n+6\epsilon{m}+48$ such tetrahedra when $n$ is large enough. Hence, for any large enough $n$, the number of edges of $\mathcal{E}^\star$ incident to at least three tetrahedra in $\mathcal{T}$ cannot be greater than the sum of $128n+6\epsilon{m}+60$ and $96n+6\epsilon{m}+48$, as desired.
\end{proof}

Let us now show that the number of faces of $\Xi$ that are not incident to any of the tetrahedra from $\mathcal{T}^\star$ is ultimately at least a linear function of $n$ and $m$. This involves proving that most of the tetrahedra contained in $\mathcal{T}^\star$ are incident to $o$ which we do as a consequence of Lemmas~\ref{PW2.sec.6.lem.3} and~\ref{PW2.sec.6.lem.4}.

\begin{lem}\label{PW2.sec.6.lem.5}
Consider a positive number $\epsilon$. If $n$ is large enough, then the number of faces of $\Xi$ not incident to a tetrahedron from $\mathcal{T}^\star$ is at least
$$
(2-30\epsilon)m-544n-262\mbox{.}
$$
\end{lem}
\begin{proof}
Let us fix an index $j$ such that $1\leq{j}\leq{m-1}$ and denote
$$
U_j=\bigl\{u_j^0,\ldots,u_j^{2n}\bigr\}\cup\bigl\{u_{2m+1-j}^0,\ldots,u_{2m+1-j}^{2n}\bigr\}\mbox{.}
$$

Further denote by $\mathcal{S}_j$ the subset made of the edges in $\mathcal{E}^\star$ with one extremity in $U_j$ and the other in $U_{j-1}$ with the convention that
$$
U_0=\bigl\{a^0,\ldots,a^{2n}\}\cup\{b^0,\ldots,b^{2n}\bigr\}\mbox{.}
$$

The faces of $\Xi$ incident to at least one edge from $\mathcal{S}_j$ form one of the $m-1$ oblique strips of equilateral triangles shown in Fig.~\ref{PW2.sec.6.fig.2}. 
It follows from Lemmas~\ref{PW2.sec.6.lem.3} and~\ref{PW2.sec.6.lem.4}, that if $n$ is large enough, then at most
$$
272n+15\epsilon{m}+132
$$
of the sets $\mathcal{S}_1$ to $\mathcal{S}_{m-1}$ are not entirely contained in $\mathcal{E}^\star_2$.

Observe that if $\mathcal{S}_j$ is a subset of $\mathcal{E}^\star_2$ then there are exactly $2(4n+1)$ tetrahedra in $\mathcal{T}^\star$ that are incident to some triangle in the corresponding strip and that all of them must share a vertex. Denote by $\mathcal{T}^\star_j$ the set of these tetrahedra. One of these tetrahedra has two vertices in $R^1$ and another has two vertices in $R^{2n}$. One can see, revisiting Fig.~\ref{PW2.sec.5.fig.1} that $o$ is the only possible common vertex for the two corresponding tetrahedra of $K_{ii}$, and therefore the only possible common vertex for all tetrahedra in $\mathcal{T}^\star_j$. In other words, if $\mathcal{S}_j$ is a subset of $\mathcal{E}^\star_2$, then all the tetrahedra in $\mathcal{T}^\star_j$ admit $o$ as a vertex. 

In order to prove the lemma, it suffices to show that at most
$$
544n+30\epsilon{m}+264
$$
of the $2(m+1)$ faces of $\Xi$ obtained by embedding a triangle from $\Gamma$ in $\mathbb{H}^3$ can be incident to a tetrahedron from $\mathcal{T}^\star$. Consider such a face $\tau$ of $\Xi$ and denote by $\sigma$ the tetrahedron in $\mathcal{T}$ that is incident to $\tau$. Observe that in order for $\sigma$ to belong to $\mathcal{T}^\star$, that tetrahedron must be incident to a lower face of $\Xi$ at an end of one of the oblique strip of equilateral triangles shown in Fig.~\ref{PW2.sec.6.fig.2}. Let $\mathcal{S}_j$ be the set of edges corresponding to that strip. One can see in Fig.~\ref{PW2.sec.5.fig.1} that $o$ cannot be a vertex of a triangle from $\Gamma$ and therefore it cannot be a vertex of $\tau$ or $\sigma$. By the above, $\mathcal{S}_j$ then cannot be a subset of $\mathcal{E}^\star_2$. As each strip has two ends, the number of the faces of $\Xi$ that correspond with a triangle from $\Gamma$ and are incident to a tetrahedron from $\mathcal{T}^\star$ cannot exceed twice the number of the sets $\mathcal{S}_1$ to $\mathcal{S}_{m-1}$ that are not a subset of $\mathcal{E}^\star_2$, as desired.
\end{proof}

We are ready to finally prove Theorem~\ref{PW2.sec.4.lem.2}

\begin{proof}[Proof of Theorem \ref{PW2.sec.4.lem.2}]
It suffices to show that, when $n$ and $m/n$ are large enough, $K$ contains two different arcs with the same pair of vertices. Assuming otherwise and setting $\epsilon$ to $1/30$, it follows from Lemma \ref{PW2.sec.6.lem.3.5} that
\begin{equation}\label{PW2.sec.4.lem.2.eq.1}
|\mathcal{T}\mathord{\setminus}\mathcal{T}^\star|\leq\frac{m}{10}+64n+30
\end{equation}
when $n$ is large enough. However, according to Lemma \ref{PW2.sec.6.lem.5}, the number of faces of $\Xi$ that are not incident to a tetrahedron from $\mathcal{T}^\star$ is least $m-544n-238$. As no vertex of $\Xi$ is contained in less than four faces of $\Xi$, every tetrahedron of $\mathcal{T}$ is incident to at most two faces of $\Xi$. As a consequence,
\begin{equation}\label{PW2.sec.4.lem.2.eq.2}
|\mathcal{T}\mathord{\setminus}\mathcal{T}^\star|\geq\frac{m}{2}-272n-131
\end{equation}
when $n$ is large enough. Fixing $n$ and letting $m$ go to infinity in (\ref{PW2.sec.4.lem.2.eq.1}) and (\ref{PW2.sec.4.lem.2.eq.2}) results in a contradiction. In particular, when $n$ and $m/n$ are both large enough, $K$ must contain two different arcs with the same vertex pair.
\end{proof}

\medskip
\noindent{\bf Acknowledgement.} The first author is partially supported by the ANR project SUGAR (Surfaces, G{\'e}ométrie et Algorithmes), grant number ANR-25-CE40-0416 and the second author by the National Natural Science Foundation of China (Grant No. 12301432) and the Natural Science Foundation of Guangdong Province (Grant No. 2023A1515010658). The manuscript was partly written while the second author was visiting the first one thanks to funding of the MathSTIC (CNRS FR3734) research consortium.

\bibliography{FlipGraphNonConvexity}
\bibliographystyle{ijmart}

\end{document}